\documentclass{amsart}

\usepackage{a4wide}
\usepackage{graphicx}
\usepackage{multirow}

\newtheorem{theorem}{Theorem}[section]
\newtheorem{lemma}[theorem]{Lemma}
\newtheorem{proposition}[theorem]{Proposition}

\newtheorem*{theorem*}{Theorem}
\newtheorem*{theo}{Theorem}

\theoremstyle{definition}
\newtheorem{definition}[theorem]{Definition}

\theoremstyle{remark}
\newtheorem{remark}[theorem]{Remark}

\numberwithin{equation}{section}

\newcommand {\Di}  {{\mathcal D}}

\newcommand {\T}  {{\mathcal T}}

\author{Shigeki Akiyama}
\address[Shigeki Akiyama]{Institute of Mathematics, University of Tsukuba, 1-1-1 Tennodai\\
Tsukuba, Ibaraki, Japan (zip:350-8571)}
\email{akiyama@math.tsukuba.ac.jp}
\author{Beno\^it Loridant}
\author{J\"org Thuswaldner}
\address[Beno\^it Loridant, J\"org Thuswaldner]{Montanuniversit\"at Leoben,
    Franz Josefstrasse 18, Leoben 8700, Austria} 
		\email{benoit.loridant@unileoben.ac.at, joerg.thuswaldner@unileoben.ac.at}
		
\thanks{The authors were supported by the project P27050 of the Austrian Science Fund (FWF), the project I1136 of the French National Agency for Research (ANR) and the FWF, and by the project I3346 of the Japan Society for the Promotion of Science (JSPS) and the FWF}

\title{Topology of planar self-affine tiles with collinear digit set}
 
\date{\today}

\dedicatory{\large Dedicated to Peter Kirschenhofer on the occasion of his 60th birthday.}

\begin{document}
\begin{abstract}
We consider the self-affine tiles with collinear digit set defined as follows. 
Let $A,B\in\mathbb{Z}$ satisfy $|A|\leq B\geq 2$ and  $M\in\mathbb{Z}^{2\times2}$ be an integral matrix with characteristic polynomial $x^2+Ax+B$. Moreover, let $\Di=\{0,v,2v,\ldots,(B-1)v\}$ for some $v\in\mathbb{Z}^2$ such that $v,M v$ are linearly independent. We are interested in the topological properties of the self-affine tile $\T$ defined by $M\T=\bigcup_{d\in\mathcal{D}}(\T+d)$. Lau and Leung proved that $\T$ is homeomorphic to a closed disk if and only if $2|A|\leq B+2$. In particular, $\T$ has no cut point. We prove here that $\T$ has a cut point if and only if $2|A|\geq B+5$. For $2|A|-B\in \{3,4\}$, the interior of $\T$ is disconnected and the closure of each connected component of the interior of $\T$ is homeomorphic to a closed disk.
\end{abstract}
\maketitle

\begin{section}{Introduction and statement of the theorems}

Let $M$ be a $d\times d$ integral \emph{expanding} matrix, \emph{i.e.}, with eigenvalues greater than $1$ in modulus, and $\Di\subset\mathbb{Z}^d$ a finite set. 
Then there is a unique nonempty compact set $\T=\T(M,\Di)$ satisfying
\begin{equation}\label{SATile}M \T=\bigcup_{a\in \Di}(\T+a)
\end{equation}
(see \cite{Hutchinson81}). 
Suppose that $\Di\subset \mathbb{Z}^d$ is a complete residue system of $\mathbb{Z}^d/M\mathbb{Z}^d$. Then $\T$ has positive Lebesgue measure (see~\cite{LagariasWang96a}) and is called \emph{integral self-affine tile with digit set} $\Di$. 
Moreover, there is a subset $\mathcal{J}$ of $\mathbb{Z}^d$ 
such that $\T+\mathcal{J}$ is a \emph{tiling of} $\mathbb{R}^d$~:
$$\bigcup_{s\in\mathcal{J}}(\T+s)=\mathbb{R}^d \;\;\textrm{and }\;\lambda_d((\T+s)\cap(\T+s'))=0\textrm{ if } s\ne s'\in\mathcal{J},
$$
where $\lambda_d$ is the $d$-dimensional Lebesgue measure (see~\cite{LagariasWang97}). If $\mathcal{J}=\mathbb{Z}^d$, $\T$ is called a \emph{self-affine $\mathbb{Z}^d$-tile}.

Self-affine tiles are thoroughly studied in the literature~\cite{Kenyon92,GroechenigHaas94,LagariasWang96a,Solomyak97}. Their topological properties have numerous connections with arithmetical properties of numeration systems~\cite{Thurston89,Katai95}, and more generally with dynamical properties of discrete dynamical systems~\cite{Rauzy82,Praggastis92}. Several topological properties of self-affine tiles have already been investigated: connectedness~\cite{KiratLau00},  homeomorphy to the closed disk~\cite{LuoRaoTan02,BandtWang01}, interior components~\cite{NgaiNguyen03,NgaiTang05} and fundamental group~\cite{LuoThuswaldner06}. 

This paper is devoted to \emph{cut points} of a class of self-affine tiles. A cut point of a connected set $T$ is a point $x\in T$ such that $T\setminus \{x\}$ is no longer connected. The study of cut points and more generally of \emph{cut sets} (when $\{x\}$ is replaced by a set $X\subset T$) is of great importance for the understanding of fractal sets with a wild topology. They were used to give a combinatorial description of the fundamental group of one-dimensional spaces~\cite{DorferThuswaldnerWinkler13} (see also~\cite{AkiyamaDorferThuswaldnerWinkler09} for the case of the Sierpinski gasket). On the opposite, the lack of cut points has an impact on the boundary of the complement of locally connected plane continua~\cite{Whyburn79}. This was exploited in~\cite{NgaiTang05} to show the homeomorphy to the closed disk  of interior components of some self-affine tiles. Algorithms for finding cut sets of self-affine tiles can be found in~\cite{LoridantLuoSellamiThuswaldner16}.

We are interested in a class of self-affine tiles with collinear digit set defined as follows. Let $M\in\mathbb{Z}^{2\times 2}$ with $\det(M)=B\geq 2$ and suppose that $M$ has characteristic polynomial $x^2+Ax+B$. Then $M$ is expanding iff $|A|\leq B$. Let $\Di=\{0,v,2v,\ldots,(B-1)v\}$ for some $v\in\mathbb{Z}^2$ such that $v,M v$ are linearly independent. Leung and Lau proved in~\cite{LauLeung07} that the associated self-affine tile $\T(M,\Di)$ is homeomorphic to the closed disk if and only if $2|A|\leq B+2$. In this paper, we are able to give topological properties also for the reverse inequality and establish the following theorem which treats topological properties for all planar self-affine tiles with collinear digit set and expanding matrix $M\in\mathbb{Z}^{2\times 2}$ with positive determinant.
\begin{theo}\label{th:cpnocp} Let $M\in\mathbb{Z}^{2\times 2}$ with  characteristic polynomial $x^2+Ax+B$, where $|A|\leq B\geq 2$.  Let $\Di=\{0,v,2v,\ldots,(B-1)v\}$ for some $v\in\mathbb{Z}^2$ such that $v,M v$ are linearly independent. Denote by $\T=\T(M,\Di)$ the associated self-affine tile. Then the following holds.
\item[$(i)$] For $2|A|-B\leq 2$, $\T$ is homeomorphic to the closed disk.
\item[$(ii)$] For $2|A|-B\in\{3,4\}$, $\T$ has no cut point but its interior is disconnected. The closure of each connected component of the interior of $\T$ is homeomorphic to the closed disk.
\item[$(iii)$] For $2|A|-B \ge 5$, $\T$ has a cut point.
\end{theo}

As mentioned above, Item $(i)$ has already been proved in \cite{LauLeung07}. The objective of the present paper is thus to show Items $(ii)$ and $(iii)$.\\

The paper is organized as follows. In Section~\ref{sec:reduccns}, we show that we can restrict the topological study of the class of self-affine tiles with collinear digit set under consideration to a subclass. In this framework, we prove Item $(iii)$ of our Theorem in Section~\ref{sec:thcp}: $\T$ has a cut point for $2|A|-B \ge 5$. We give the explicit address of a point $z$ such that $\{z\}$ is the intersection of two sets $D_1,D_2\subset\T$ having no other common points and satisfying $D_1\cup D_2=\T$. Section~\ref{sec:prep} prepares for the proof of the cases $2|A|-B\in\{3,4\}$. The special case $A=4,B=5$ was considered by Ngai and Tang in~\cite{NgaiTang04}: the associated tile has no cut point. The proof required the construction of a curve $Q$ inside the tile with specific properties. We will use this technique. However, our construction of the curve $Q$ will rely on the boundary parametrization procedure for self-affine tiles recently introduced by the first two authors in~\cite{AkiyamaLoridant11}. In particular, our curve $Q$ will be a subset of the boundary $\partial \T$ of the tile. In Section~\ref{sec:prep}, we recall the method of Ngai and Tang in the study of cut points of self-affine tiles as well as the boundary parametrization procedure. We finally come to the proof of Item $(ii)$ of our Theorem. For technical reasons, we separate the two cases: Section~\ref{sec:thnocp3} is devoted to the case  $2|A|-B=3$. The case $2|A|-B=4$ can be treated with the same tools, but with significant changes in the construction of the curve $Q$. The differences to the case $2|A|-B=3$ are explained in Section~\ref{sec:thnocp4} and the complete proof, which is rather long, will be reproduced in a separate paper~\cite{Loridant0000}. 
 
\begin{remark}\label{rem:toth}
Our theorem applies to the subclass of so-called \emph{quadratic canonical number system tiles}, \emph{i.e.}, for the parameters $A,B\in\mathbb{Z}$ such that  $-1\leq A\leq B\geq 2$ and $x^2+Ax+B$ is irreducible. 

Leung and Lau included the case $B\leq -2$ in their study~\cite{LauLeung07}. Then the matrix $M$ is expanding if and only if $|A|\leq |B+2|$. The collinear digit set  has the form $\Di=\{0,v,2v,\ldots,(|B|-1)v\}$. Leung and Lau showed that $\T$ is homeomorphic to the closed disk if and only if $2|A|\leq |B+2|$. The techniques used in our paper would apply to the case $B\leq -2$. The boundary parametrization for $B\leq -2$ was performed in~\cite{AkiyamaLoridant10}. However, since the involved computations would be very lengthy, we prefer to restrict to the case $B\geq 2$. 
\end{remark}

\begin{remark}\label{conj}
In the literature, several examples of self-affine tiles happen to have finitely many interior components up to affinity. This property was proved for the Levy dragon in~\cite{Alster10} and the tile $\T$ of our class associated to the parameters $A=4,B=5$ in~\cite{BernatLoridantThuswaldner10}. In general, one may wonder whether every self-affine attractor has finitely many homeomorphy classes of interior components.    
\end{remark}

\begin{figure}[h]
\begin{tabular}{cc}
\includegraphics[width=4cm,height=4cm]{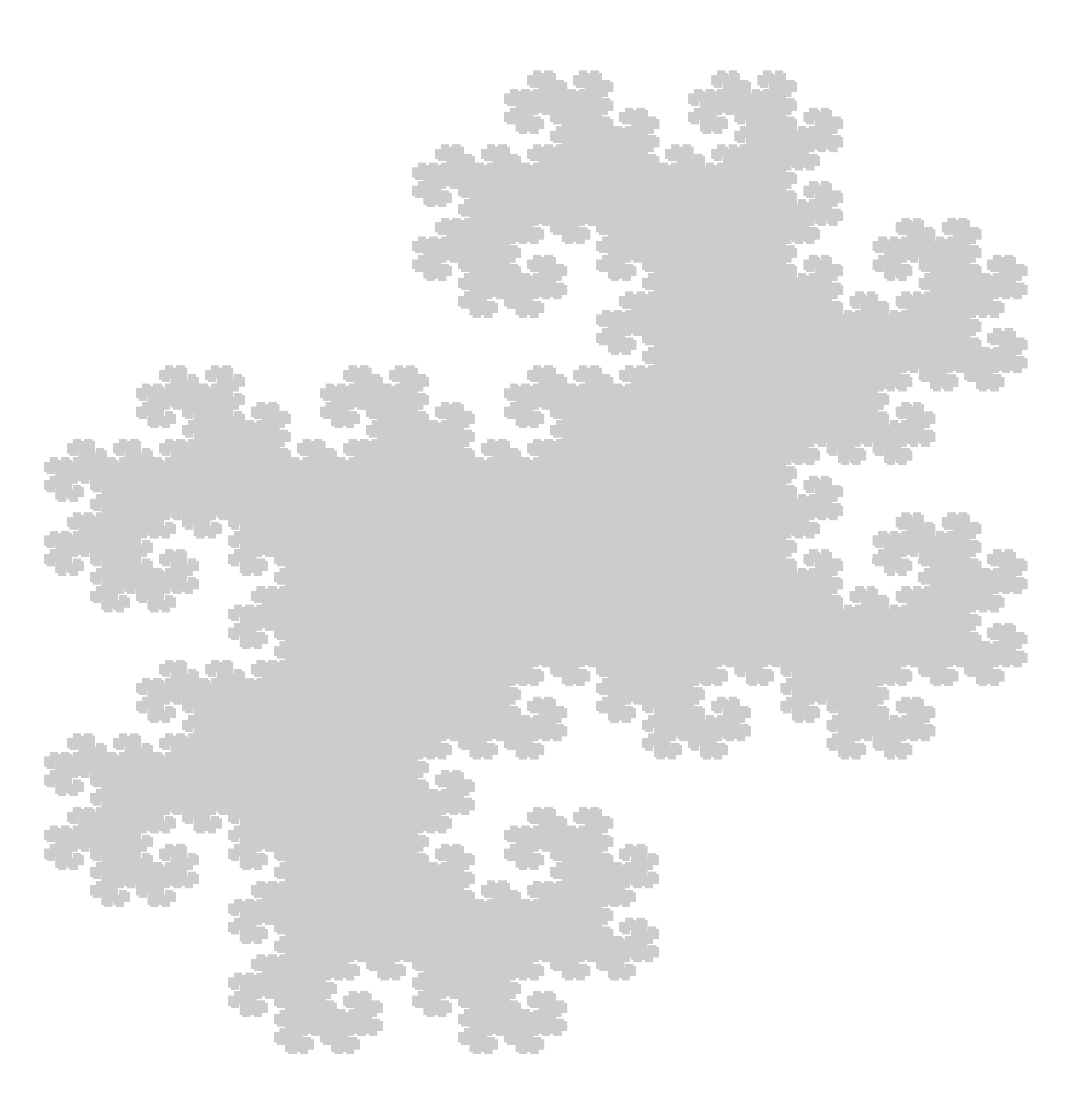}&\includegraphics[width=4cm,height=4cm]{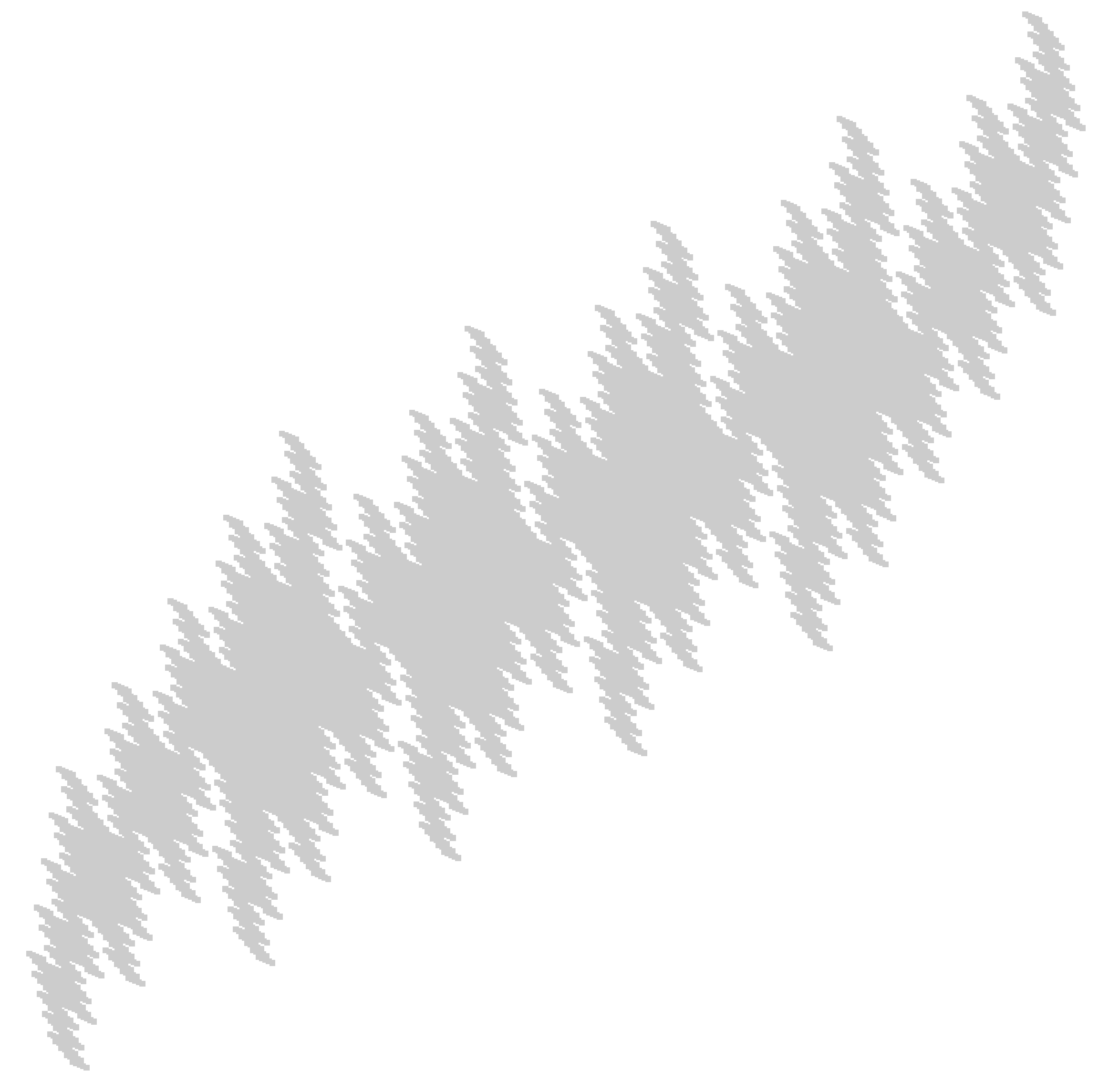}
\\
The Knuth dragon is disk-like.& Case $A=4,B=5$: this tile has no cut point.\\
\multicolumn{2}{c}{\includegraphics[width=5cm,height=6cm]{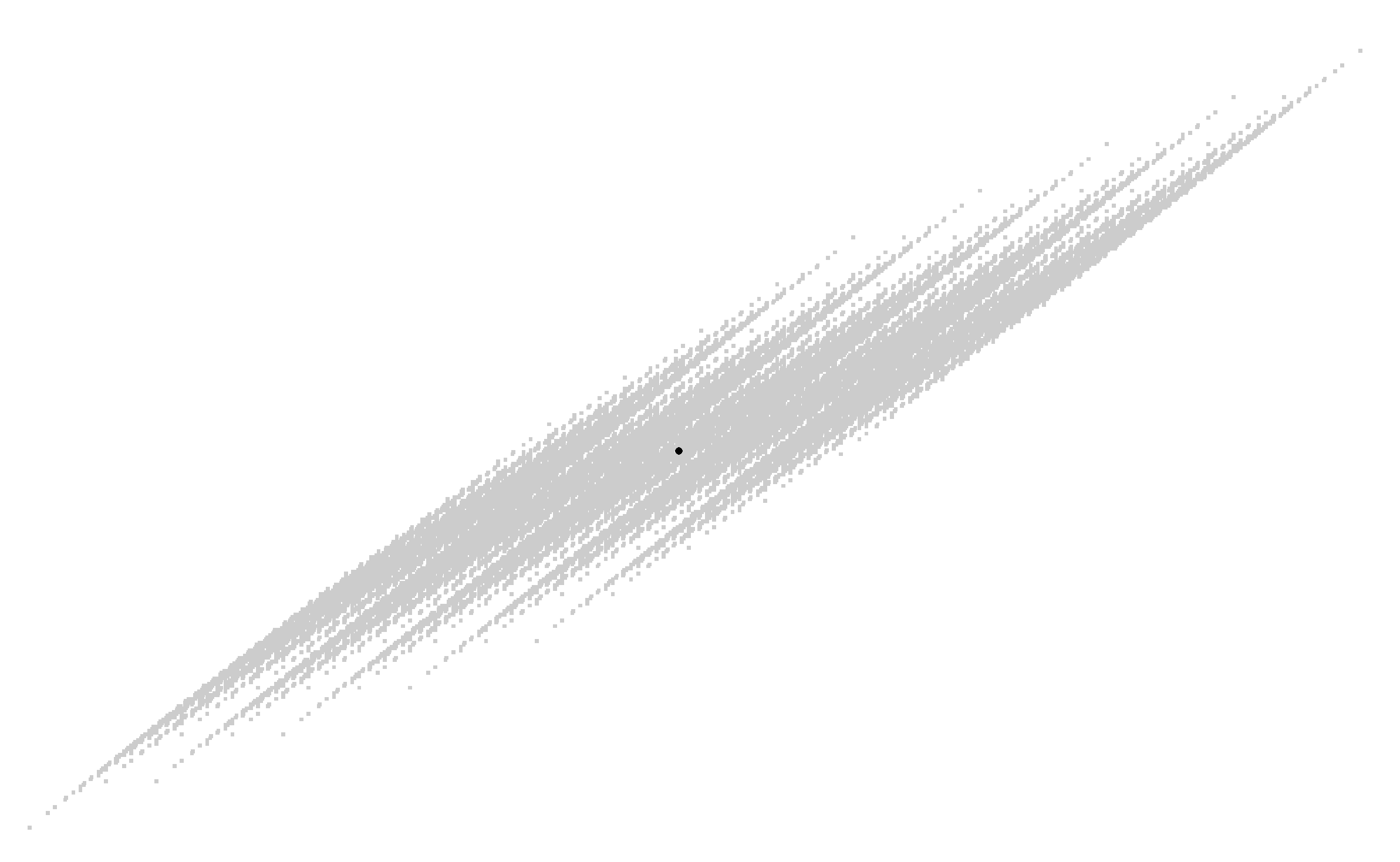}}\\
\multicolumn{2}{c}{Case $A=5,B=5$: this tile has a cut point.}

 \end{tabular}
 \caption{Several examples of tiles} \label{fig:Examples}
\end{figure}

\end{section}

\begin{section}{Restriction to the subclass $0< A\leq B\geq 2$}\label{sec:reduccns}
The following considerations can be found in~\cite{AkiyamaLoridant10}.
Let $A,B\in\mathbb{Z}$ with $|A|\leq B\geq 2$, $M_0\in\mathbb{Z}^{2\times 2}$ with characteristic polynomial $x^2+Ax+B$ and $\Di_0=\{0,v,2v,\ldots,(B-1)v\}$ for some $v\in\mathbb{Z}^2$ such that $v,M_0 v$ are linearly independent.
The self-affine tile $\T(M_0,\Di_0)$ is an affine transformation of the self-affine $\mathbb{Z}^2$-tile $\T(M,\Di)$ defined by~:
\begin{equation}\label{CollClass}
M=\left(\begin{array}{cc}0&-B\\1&-A\end{array}\right),
\;\Di=\left\{\left(\begin{array}{c}0\\0\end{array}\right),\ldots,\left(\begin{array}{c}B-1\\0\end{array}\right)\right\}.
\end{equation}  
Indeed, denote by $C$ the matrix of change of basis from the canonical basis to $(v,M_0 v)$. Then the relations
$$M=C^{-1}M_0C,\;\;\Di_0= C\Di
$$
hold and it follows that $\T(M_0,\Di_0)=C\T(M,\Di)$, thus these tiles have the same topological properties. As for $A=0$ the tile is just a rectangle, we will further suppose $A\ne 0$. Finally, we also mentioned in~\cite{AkiyamaLoridant10}  that changing $A$ to $-A$ corresponds to a reflection followed by a translation. This allows to restrict to the case $A>0$. Indeed, let 
\begin{equation}\label{switchdata}
P=\left(\begin{array}{cc}
1&0\\0&-1
\end{array}\right),\;\;
M_1=\left(\begin{array}{cc}
0&-B\\1&-A
\end{array}\right),
\;\;
M_2=\left(\begin{array}{cc}
0&-B\\1&A
\end{array}\right)
\end{equation}
and $\Di$ as in~(\ref{CollClass}). Moreover, let $\T_1:=\T(M_1,\Di)$ and $\T_2:=\T(M_2,\Di)$. Then one can check that
$PM_1P^{-1}=-M_2$, $P\Di=\Di$ and 
\begin{equation}\label{T12}
\T_2=P \T_1+\underbrace{\sum_{i\geq 0}M_2^{-2i-1}\left(\begin{array}{c}B-1\\0\end{array}\right)}_{=:\mathbf{v}}.
\end{equation}
In particular, the tiles $\T_1$ and $\T_2$ have the same topological properties.  By the above considerations, it is sufficient to prove the theorem for the case $0< A\leq B\geq 2$.

From now on, the tile $\T$ is the planar integral self-affine tile satisfying~\eqref{SATile} for the data $(M,\mathcal{D})$ given in~\eqref{CollClass}. It  can be explicitly written as
\begin{displaymath}
\mathcal{T}=\overline{\left\{\sum_{i=1}^nM^{-i}a_i;\
 a_i\in\mathcal{D},n\in\mathbb{N}\right\}}
\end{displaymath}
and gives rise to a tiling of $\mathbb{R}^2$. In other words, $\T$ is the closure of its interior and $\mathbb{R}^2$ is the non-overlapping union of the $\mathbb{Z}^2$-translates of $\T$:
$$\mathbb{R}^2=\bigcup_{x\in\mathbb{Z}^2}(\T+x)\;\textrm{ and }\lambda_2((\T+x)\cap(\T+x'))=0 \;(x\ne x'),
$$ 
where $\lambda_2$ is the two-dimensional Lebesgue measure (see \emph{e.g.}~\cite{AkiyamaThuswaldner05}).

\end{section}

\begin{section}{$\T$ has a cut point for $2A-B\geq 5$}\label{sec:thcp}

To prove Item $(iii)$ of our theorem, we will write down two compact subsets $D_1,D_2$ of $\T$ with $\T=D_1\cup D_2$ and show that $D_1\cap D_2$ consists of a single point, a cut point of $\T$. The assumption $2A-B\geq 5$ will be used to show that natural subdivisions of $\T$ divided among $D_1$ and $D_2$ are far enough from each other, thus do not intersect (see Proposition~\ref{1point}). 

 Let $-1\leq A\leq B\geq 2$. We define the \emph{set of neighbors} $\mathcal{S}$ by
\begin{displaymath}
\mathcal{S}=\left\{s\in\mathbb{Z}^2;\ s\neq 0,\ \,
\mathcal{T}\cap(\mathcal{T}+s)\neq \emptyset\right\}.
\end{displaymath}

It was computed in~\cite{AkiyamaThuswaldner05}. Let
\begin{equation}\label{eq:neigh}
P_n=\left(\begin{array}{c}n-(n-1)A\\-(n-1)\end{array}\right),\quad
Q_n=\left(\begin{array}{c}-n+nA\\n\end{array}\right),\quad
R=\left(\begin{array}{c}-A\\-1\end{array}\right),\qquad n\geq 1,
\end{equation}
then the set of neighbors consists of  $2+4J$ elements:
\begin{equation}\label{eq:neighset}
\mathcal{S}=\left\{\pm P_1,\ldots,\pm P_J,\pm Q_1,\ldots,\pm Q_J,\pm R\right\}.
\end{equation}
Here, 
\begin{equation}
\label{J}
J=\max\left\{\textstyle 1,\left\lfloor\frac{B-1}{B-A+1}
\right\rfloor\right\}.
\end{equation}
Note that
\begin{displaymath}
J>1 \qquad \text{iff} \qquad 2A-B\geq 3.
\end{displaymath}

We assume throughout this section that 
\begin{equation}\label{eq:ass}
2A-B\geq 5.
\end{equation}

\textbf{Notations.} Let $l\in\mathbb{N},m\in\mathbb{N}\cup\{\infty\}$, and digits $a_{-l},\ldots,a_{-1},a_0,a_1,\ldots,a_m\in\{0,\ldots,B-1\}$, we write 
\begin{equation}\label{eq:not}
a_{-l}\ldots a_{-1}a_0\;._M\; a_{1}a_2\ldots a_m \textrm{ or simply }a_{-l}\ldots a_{-1}a_0\;.\; a_{1}a_2\ldots a_m
\end{equation}  
for the point $\sum_{i=-l}^m M^{-i}\left(\begin{array}{c}a_i\\0\end{array}\right)\in\mathbb{R}^2$. If $m=0$, we just write $a_{-l}\ldots a_{-1}a_0$. Moreover, for digits $a_1,\ldots,a_m\in\{0,\ldots,B-1\}$ and $p\geq 0$, 
\begin{equation}\label{eq:not2}
(a_1\ldots a_m)^p=\underbrace{a_1\ldots a_m\cdots a_1\ldots a_m}_{p\textrm{ times}}
\end{equation}
 denotes the  $p$ successive concatenations of  $a_1\ldots a_m$. If $p=\infty$, we just write 
\begin{equation}\label{eq:not3}
\overline{a_1\ldots a_m}.
\end{equation}
Also, we denote by $a'$ the digit $B-1-a$. Moreover, we set
\[
a^{(n)} := \begin{cases} a, & n\equiv 0 \bmod 2\\
B-1-a, & n\equiv 1 \bmod 2.\\
\end{cases}
\]

Let us define

\begin{equation}\label{eq:D1}
D_1 :=\left \{ \sum_{j=1}^\infty M^{-j}\left(\begin{array}{c}a_j\\0\end{array}\right) \,;\,a_j\in\{0,\ldots,B-1\}, 
a_{1}^{(0)}a_{2}^{(1)}a_{3}^{(2)}
\ldots
\le_{lex} (A-3)(A-3)(A-3)\ldots \right\}
\end{equation}

and
\begin{equation}\label{eq:D2}
D_2 :=\left \{ \sum_{j=1}^\infty M^{-j}\left(\begin{array}{c}a_j\\0\end{array}\right) \,;\,\,a_j\in\{0,\ldots,B-1\}, 
a_{1}^{(0)}a_{2}^{(1)}a_{3}^{(2)}
\ldots
\ge_{lex} (A-3)(A-3)(A-3)\ldots \right\}.
\end{equation}
These sets can easily be written as graph directed sets. The corresponding graph is depicted in Figure~\ref{fig:GIFS} in the sense of~\cite{MauldinWilliams88}. In particular, $D_1$ and $D_2$ are closed compact subsets of $\T$ satisfying
$$
\T=D_1\cup D_2.
$$

\begin{figure}
\includegraphics[width=110mm,height=50mm]{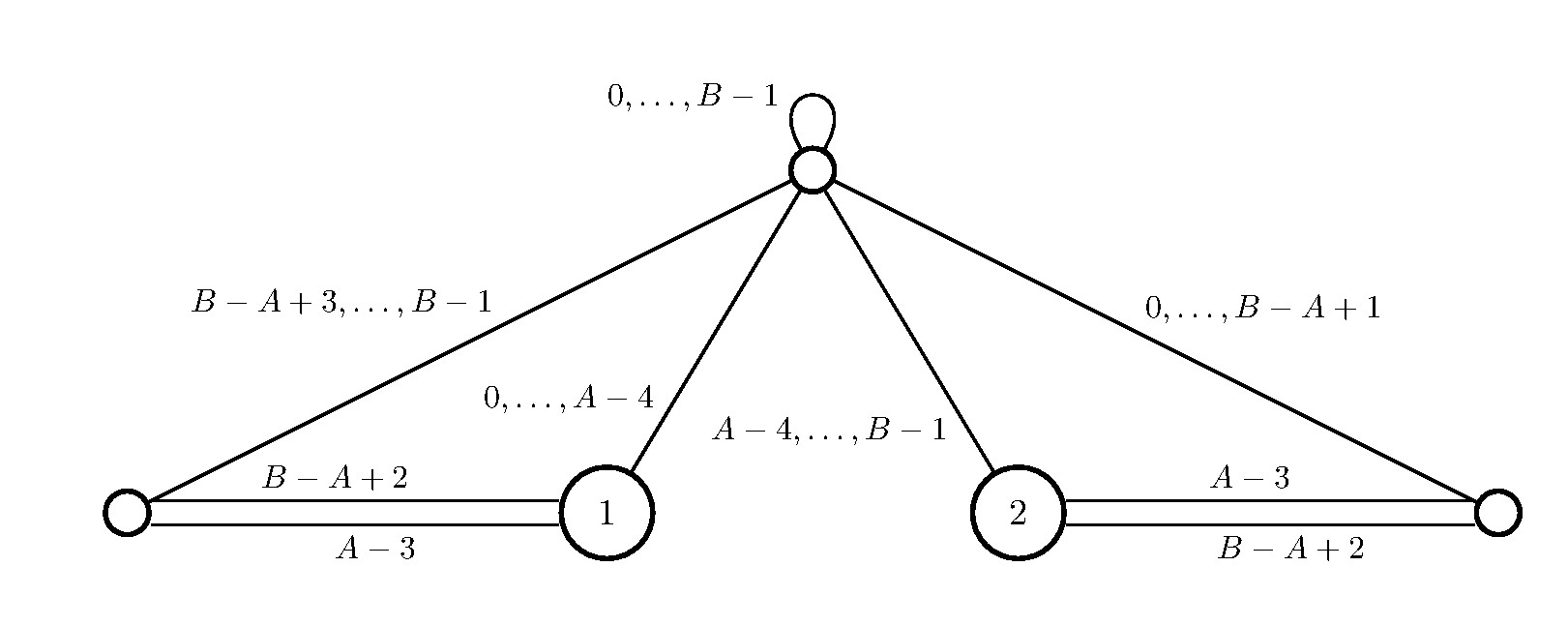}
\caption{$D_i$ is the solution of the above GIFS associated to the state $i$ ($i=1,2$)}
\label{fig:GIFS}
\end{figure}

\begin{proposition}\label{1point}
We have
\[
D_1 \cap D_2 = \left\{
0.(A-3)(A-3)'(A-3)(A-3)'\ldots\right\}
\]
In particular, this intersection consists of a single point.
\end{proposition}

\begin{proof}

We are interested in points where $D_1$ and $D_2$ coincide. Let us set $I := D_1 \cap D_2$. We have
\begin{eqnarray*}
D_1 &\subset&  M^{-1}[\T \cup (\T + 1) \cup \cdots \cup (\T +(A-3))], \\
D_2 &\subset&  M^{-1}[(\T + (A-3)) \cup \cdots \cup (\T +(B-1))]. \\
\end{eqnarray*}

Thus, in order to be in $I$, a point $x$ has to be located in one 
of the sets
\[
(\T +a) \cap (\T+ b), \qquad 0\leq a \le A-3 \le b\leq B-1.
\]
This intersection is empty unless $\left(\begin{array}{c}a-b\\0\end{array}\right)\in\mathcal{S}$ is a neighbor of $\T$. By~(\ref{eq:neigh}) and~\eqref{eq:neighset}, 
this is true only for $(a-b) \in \{0,\pm 1\}$. Hence $I$ has to be 
contained in 
\[
M^{-1} [(\T+A-4) \cup (\T+A-3) \cup (\T+A-2)]  =: G_0
\]

We proceed to an induction proof. For $n\geq 0$, we will use the abbreviation 
$$S_n := (A-3)^{(0)}\ldots(A-3)^{(n-1)},$$ $S_0$ being the empty word. Also, for $n\geq 0$, let 
\begin{equation}\label{ia}
G_n:= M^{-n-1}[ (\T+S_n(A-4)^{(n)} ) \cup (\T+S_n(A-3)^{(n)}) \cup 
(\T+S_n(A-2)^{(n)})]  .
\end{equation}
By the previous lines, $I\subset G_0$. 
Assume now  that  $I\subset G_n$ for some $n\geq 0$.  We want to prove that $I\subset G_{n+1}$.

Suppose that $n$ is even. By the set equation~(\ref{SATile}), we can write $G_n$ as 
\[
G_n = \bigcup_{j\in \{0,\ldots,B-1\}} M^{-n-2} [(\T+ S_n(A-4)j' )
\cup (\T+S_n(A-3)j' ) 
\cup (\T+S_n(A-2)j')].
\]
By the definition of $D_1$ and $D_2$ we conclude that 
\begin{eqnarray*}
 D_1 &\cap& G_n \subset \bigcup_{j\in \{0,\ldots,B-1\}} M^{-n-2} (\T+ S_n(A-4)j')   \\
&&\cup 
\bigcup_{j\in \{0,\ldots,A-3\}}M^{-n-2} (\T+S_n(A-3)j') , \\
 D_2 &\cap& G_n \subset \bigcup_{j\in \{A-3,\ldots,B-1\}} M^{-n-2} (\T+(A-3)j')
\\
&&\cup 
\bigcup_{j\in \{0,\ldots,B-1\}} M^{-n-2}(\T+S_n(A-2)j').
\end{eqnarray*}
Thus any point of $I$  has to be contained in some
\begin{equation}\label{ab}
M^{-n-2}[(\T + a) \cap (\T+ b)]
\end{equation}
where 
\begin{eqnarray*}
a \in \{     S_n(A-4)j' ;\;       0\le j\le B-1     \} \cup \{ 
S_n(A-3)j';\;  0\le j\le A-3  \}, \\
b \in \{     S_n(A-2)j' ;\;       0\le j\le B-1     \} \cup \{ 
S_n(A-3)j';\;  A-3 \le j\le B-1  \}. \\
\end{eqnarray*}

We want to show that the intersection in~\eqref{ab} is empty 
unless 
\begin{eqnarray}\label{aabb}
a &\in& \{ S_n(A-3)(A-4)', S_n(A-3)(A-3)' \} \quad\hbox{and}  \nonumber\\
b &\in& 
\{ S_n(A-3)(A-3)', S_n(A-3)(A-2)' \}. 
\end{eqnarray}
If this is proved we are ready because this implies that 
$I\subset G_{n+1}$.

The intersection in \eqref{ab} is nonempty if and only if $b-a$ equals $0$ or $b-a$
is contained in the set of neighbors given in~\eqref{eq:neigh} and~\eqref{eq:neighset}.

First let $a=S_n(A-4)j$ and 
$b=S_n(A-2)k$ with $j,k \in \{0,\ldots, B-1\}$. Then
\[
b-a=\left(\begin{array}{c}k-j\\0\end{array}\right)
+M\left(\begin{array}{c}2\\0\end{array}\right)=\left(\begin{array}{c}k-j\\2\end{array}\right).
\]
Note that $k-j \le B-1$. However, for $b-a$ to be a neighbor, we must have 
$$k-j \in \{-3+2A, 
-2+2A\}.$$ 
By our main assumption~\eqref{eq:ass}, we have $-2+2A>-3+2A\geq B+2$. This is impossible.

Second let $a=S_n(A-4)j$ and $b=S_n(A-3)k$ with $j \in 
\{0,\ldots,B-1\}$ and $k \in \{0, \ldots, \underbrace{B-A+2}_{=B-1-(A-3)}\}$. Then
\[
b-a=\left(\begin{array}{c}k-j\\0\end{array}\right)
+M\left(\begin{array}{c}1\\0\end{array}\right)=\left(\begin{array}{c}k-j\\1\end{array}\right). 
\]
Note that $k-j \le B-A+2$. However, for $b-a$ to be a neighbor, we must have 
$$k-j \in \{-2+A, -1+A, 
A\}.$$
By our main assumption~\eqref{eq:ass}, we have $A>-1+A>-2+A\geq B-A+3 $. This is impossible.

Third let $a=S_n(A-3)j$ and $b=S_n(A-2)k$ with $j \in 
\{B-A+2,\ldots, B-1\}$ and $k \in \{0, \ldots B-1\}$. Then
\[
 b-a=\left(\begin{array}{c}k-j\\0\end{array}\right)
+M\left(\begin{array}{c}1\\0\end{array}\right)=\left(\begin{array}{c}k-j\\1\end{array}\right). 
\]
Note that $k-j \le A-3$. However, for $b-a$ to be a neighbor, we must have 
$$k-j \in \{-2+A, -1+A, 
A\}.
$$ 
This is impossible.

Fourth let $a=S_n(A-3)j$ and $b=S_n(A-3)k$ with $j \in 
\{B-A+2,\ldots, B-1\}$ and $k \in \{0, \ldots, B-A+2\}$. Then 
$$b-a=\left(\begin{array}{c}k-j\\0\end{array}\right).$$  
Note that $k-j\leq 0$. Moreover,  for $b-a$ to be a neighbor, we must have $k-j\in\{\pm 1\}$. This leads to the possible pairs
$$\left\{\begin{array}{l}
k=B-A+2=(A-3)'=j\\
k=B-A+2=(A-3)',\;j=B-A+3=(A-4)'\\
k=B-A+1=(A-2)',\;j=B-A+2=(A-3)'
\end{array}\right..
$$

Therefore, the intersection in~\eqref{ab} is nonempty for 
$$\left\{\begin{array}{l}
a=S_n(A-3)(A-3)'=b\\
a=S_n(A-3)(A-4)',\;b=S_n(A-3)(A-3)'\\
a=S_n(A-3)(A-3)',\;b=S_n(A-3)(A-2)'
\end{array}\right..
$$
These are constellations of~\eqref{aabb}. It follows that any point of $I$ has to be contained in $G_{n+1}$. The case of $n$ odd can be treated in an analogous way. \\

We conclude that $I=D_1\cap D_2$ is contained in $G_n$ defined in~(\ref{ia}) for all $n\geq 0$. Since the three sets
$$M^{-n-1}[ \T+S_n(A-4)^{(n)} ],\; M^{-n-1}[\T+S_n(A-3)^{(n)}],\; 
M^{-n-1}[\T+S_n(A-2)^{(n)}] 
$$ 
all converge to the point set  $\{0.(A-3)(A-3)'(A-3)(A-3)'\ldots\}$ in Hausdorff metric, this shows that $I$ consists in a unique point.
\end{proof}

\begin{theorem}\label{th:main}
If $2A-B \ge 5$ then 
$\mathcal{T}$ has at least one cut point, namely, the point
\[
0.(A-3)(A-3)'(A-3)(A-3)'(A-3)\ldots
\]
\end{theorem}

\begin{proof}
 We define $D_1$ and $D_2$ as in~\eqref{eq:D1} and~\eqref{eq:D2}. As mentioned at the beginning of the section,  these sets are compact sets satisfying $\mathcal{T}=D_1\cup D_2$.  
 
 We denote by $z$ the point $0.(A-3)(A-3)'(A-3)(A-3)'(A-3)\ldots$ By Proposition~\ref{1point}, 
 $$D_1\cap D_2=\{z\}.
$$
Therefore, 
$$\T\setminus \{z\}=(D_1\setminus\{z\})\cup (D_2\setminus\{z\}), 
$$
where 
$$\overline{D_1\setminus\{z\}}\cap (D_2\setminus\{z\})=D_1\cap (D_2\setminus\{z\})=\emptyset $$ 
as well as 
$$\overline{D_2\setminus\{z\}}\cap (D_1\setminus\{z\})=D_2\cap( D_1\setminus\{z\})=\emptyset. 
$$ 
In other words, $z$ is a cut point of the connected set $\T$.
\end{proof}

The cut point obtained in Theorem~\ref{th:main} is depicted for $A=5,B=5$ within the tile at the bottom of Figure~\ref{fig:Examples}. 

\end{section}

\begin{section}{Two techniques in the study of self-affine tiles}\label{sec:prep} 
In this section, we recall two techniques in the topological study of self-affine tiles. The first one goes back to Ngai and Tang and aims at showing that a self-affine tile has no cut point. Recall that an \emph{iterated function system} (IFS) $\{f_i\}_{i=1}^m$ of injective contractions on $\mathbb{R}^2$ has a unique nonempty compact \emph{attractor set} $\T$ satisfying
$$\T=\bigcup_{i=1}^mf_i(\T)
$$
(\cite{Hutchinson81}). Moreover, $\{f_i\}_{i=1}^m$ satisfies the \emph{open set condition} (OSC) whenever there exists a nonempty bounded open set $U\subset \mathbb{R}^2$ such that 
$$\bigcup_{i=1}^mf_i(U)\subset U\;\textrm{ and }\;\forall i\ne j, \;f_i(U)\cap f_j(U)=\emptyset.
$$
Self-affine tiles $\T$ satisfy the open set condition by taking $U=\textrm{int}(\T)$. 
\begin{theorem}[\cite{NgaiTang04}]\label{th:nt04}
Let $\T$ be the attractor of an IFS $\{f_i\}_{i=1}^m$ of injective contractions on $\mathbb{R}^2$ satisfying the open set condition. Assume that there exists a connected subset $Q$ of $\T$ without cut points such that for each $i\in\{1,\ldots,m\}$, $\#(f_i(Q)\cap Q)\geq 2$. Then $\T$ is connected, has no cut point, and the closure of each component of $\mathrm{int}(\T)$ is homeomorphic to a closed disk.
\end{theorem}

To prove that the tile $\T$ has no cut point for $2A-B\in\{3,4\}$, we will construct a curve  $Q\subset\partial \T$ with the properties required in Theorem~\ref{th:nt04}. Our construction will rely on a second technique, the boundary parametrization of self-affine tiles introduced in~\cite{AkiyamaLoridant11}. This technique gives rise to a H\"older mapping $C:[0,1]\to\partial\T$ with $C(0)=C(1)$ that follows the fractal boundary. The boundary parametrization of the class of quadratic CNS tiles was treated as an example in~\cite{AkiyamaLoridant11}. It reads as follows. 

The boundary parametrization relies on a \emph{graph directed iterated function system} (GIFS) for the boundary in the sense of~\cite{MauldinWilliams88}. This GIFS describes the subdivision process of the boundary, as stated in Proposition~\ref{prop:primsubautomaton} below. Let $G$ be the graph depicted on the left of Figure~\ref{primsubautomaton}. We call 
$$\mathcal{R}=\{\pm P_1,\pm Q_1,\pm R\},
$$ 
as defined in Equation~(\ref{eq:neigh}). 
Moreover, for $s,s'\in\mathcal{R}$ and $a,a'\in\mathcal{D}$, there is an edge 
$$s\xrightarrow{a|a'}s'\in G:\iff M s+a'=s'+a.
$$
We sometimes simply write $s\xrightarrow{a}s'$, since $a'$ is uniquely defined by the above equation. The graph $G$ has the following properties. 
\begin{figure}
\includegraphics[width=150mm,height=80mm]{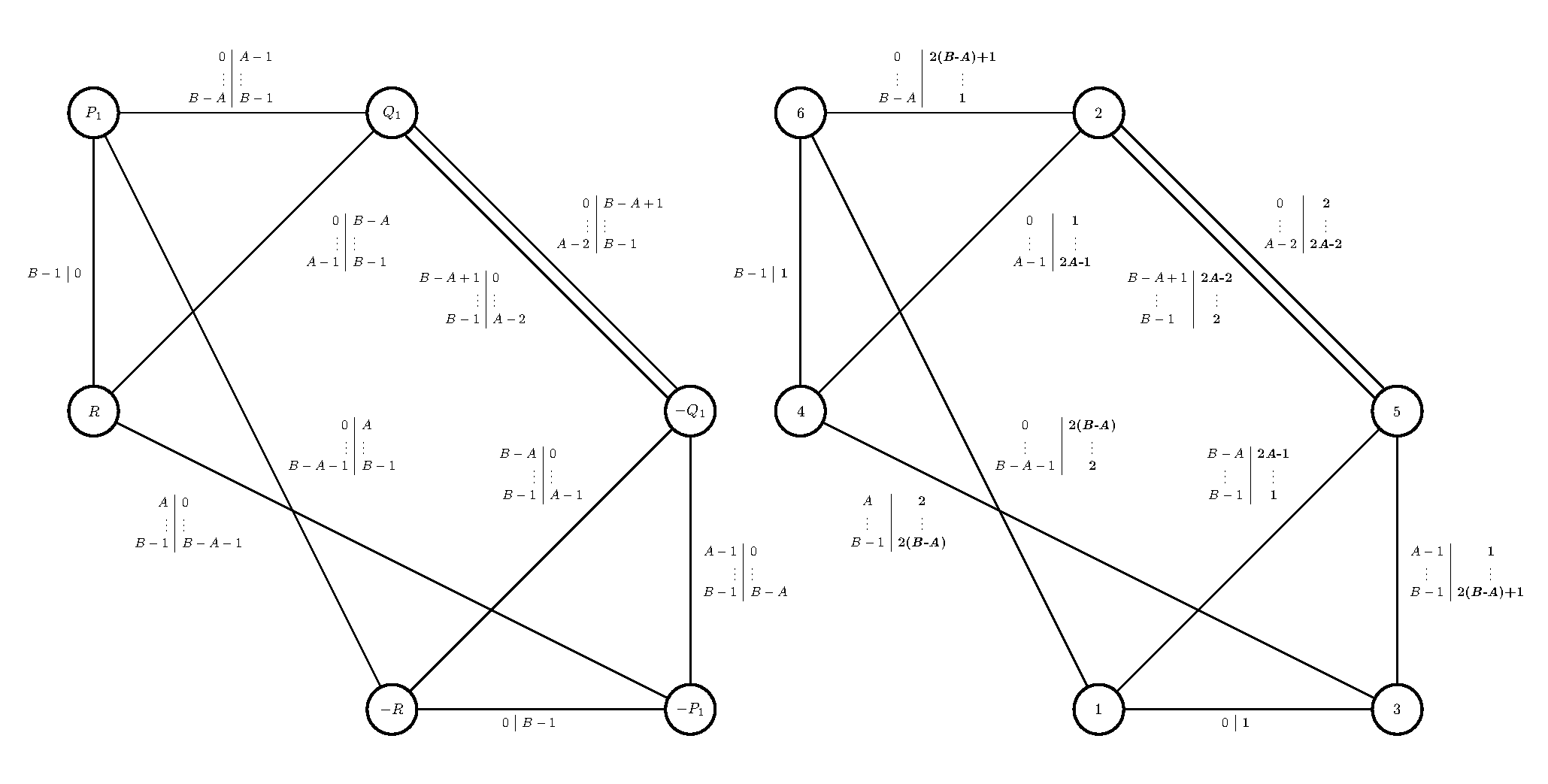}
\caption{A strongly connected automaton for the boundary (left) and its ordered extension $G^o$ (right).}
\label{CNS_DTautomaton}\label{primsubautomaton}

\end{figure}
\begin{proposition}\label{prop:primsubautomaton} Let $\T$ be quadratic a CNS-tile and $\mathbf{D}$ the incidence matrix of $G$. Then $\mathbf{D}$ is irreducible. Moreover, there exists a unique family of non-empty compact sets $(K_s)_{s\in \mathcal{R}}$ such that 
$$\left\{\begin{array}{rcl}
\partial \T&=&\bigcup_{s\in \mathcal{R}} K_s,\\\\
K_s&=&\bigcup_{s \xrightarrow{a} s'\in G}M^{-1}(K_{s'}+a).
\end{array}\right.
$$
\end{proposition}

\begin{remark}\label{rem:AeqB}
$G$ is the so-called \emph{contact graph}. It was computed in~\cite{AkiyamaThuswaldner05}. In the figure, we assume $B\geq A+1$. For $A=B$, the contact graph is obtained from the same figure by deleting the edges between between $P_1$ and $-R$ and between $-P_1$ and $R$. 
\end{remark}

$G$ being strongly connected, there is a strictly positive left eigenvector $(u_s)_{s\in\mathcal{R}}$ satisfying $\sum_{s\in\mathcal{R}}u_s=1$ associated to the  Perron Frobenius eigenvalue $\beta$ of its incidence matrix. The parametrization $C:[0,1]\to\partial \T$ is obtained by subdividing the interval $[0,1]$ proportionally to the graph $G$: in particular, a subinterval of length $u_s$ will be mapped to $K_s$ for each $s\in\mathcal{R}$. To this effect, we order the boundary pieces $K_s$ around the boundary, as well as the subpieces $\xi_{-1}(K_{s'}+\Phi(a))$ constituting $K_s$. This results in an ordered extension $G^o$, depicted on the right side of Figure~\ref{primsubautomaton}: the states are ordered from $1$ to $6$, and for each corresponding state $s^{(i)}\in\mathcal{R}$ $(i\in\{1,\ldots,6\})$, all the edges starting from $s^{(i)}$ are ordered from ${\bf 1}$ to ${\bf o_m}$. Here, ${\bf o_m}$ is the total number of edges starting from $s^{(i)}$, without reference to $i$ for sake of simplicity. In this way, the mapping 
$$\begin{array}{ccc}
G^o&\to& G\\
i\xrightarrow{a|{\bf o}}j=:(i;{\bf o})&\mapsto& s^{(i)}\xrightarrow{a}s^{(j)}
\end{array}
$$
is a bijection.  We extend this mapping to the walks of arbitrary length in $G^o$ (possibly infinite walks):
$$\begin{array}{ccc}
G^o&\to& G\\
(i;{\bf o_1},{\bf o_2},\ldots)&\mapsto& s^{(i)}\xrightarrow{a_1}s^{(j_1)}\xrightarrow{a_2}s^{(j_2)}\ldots,
\end{array}
$$
whenever $ i\xrightarrow{a_1|{\bf o_1}}j_1\xrightarrow{a_2|{\bf o_2}}j_2\ldots\in G^o$. Finally, we define the natural onto mapping 
$$\begin{array}{ccc}
\psi:G^o&\to& \partial \T\\
v&\mapsto& \sum_{n\geq 1}M^{-n}a_n,
\end{array}
$$
whenever $v:i\xrightarrow{a_1|{\bf o_1}}j_1\xrightarrow{a_2|{\bf o_2}}j_2\ldots$ is an infinite walk in $G^o$. The automaton $G^o$ induces a $\beta$-number system  $\phi^{(1)}:[0,1]\to G^o$ of Dumont-Thomas type~\cite{DumontThomas89}. In~\cite{AkiyamaLoridant11}, we proved the following result by taking $C:=\psi\circ \phi^{(1)}$.
\begin{theorem}
\label{th:boundparam}
Let $\beta$ be the spectral radius of $\mathbf{D}$. Then there exists $C:[0,1]\to\partial \T$ H\"older continuous onto mapping and a hexagon $Q\subset\mathbb{R}^2$ with the following properties. Let $\T_0:=Q$ and 
$$M\T_{n}=\bigcup_{a\in \Di} (\T_{n-1}+a).
$$
be the sequence of approximations of $\T$ associated to $Q$. Then :
\begin{itemize}
\item[$(1)$]$\lim_{n\to\infty}\partial \T_n=\partial \T$ (Hausdorff metric).
\item[$(2)$]For all $n\in\mathbb{N}$, $\partial \T_n$ is a polygonal simple closed curve.
\item[$(3)$]Denote by $V_n$ the set of vertices of $\partial \T_n$. For all $n\in\mathbb{N}$, 
$V_n\subset V_{n+1}\subset C(\mathbb{Q}(\beta)\cap [0,1])$ (\emph{i.e.}, the vertices have $\mathbb{Q}(\beta)$-addresses in the parametrization).
\end{itemize}
\end{theorem}

\begin{remark} \label{rem:Approx}The polygonal approximations $\partial\T_n$ appear in a natural way together with the parametrization. $\partial\T_0$ is obtained by joining by straight line segments the points 
$$\psi(1;\overline{\mathbf{1}}),\psi(2;\overline{\mathbf{1}}),\ldots,\psi(6;\overline{\mathbf{1}}),\psi(1;\overline{\mathbf{1}})
$$
in this order. Here, we denoted by $\overline{\mathbf{o}}$  the infinite repetition of $\mathbf{o},\mathbf{o},\ldots$. In general,  let $w_1^{(n)},\ldots, w_{m_n}^{(n)}$ be the walks of length $n$ in the automaton $G^o$, written in the lexicographical order, from $(1;\underbrace{{\bf 1},\ldots,{\bf 1}}_{n\textrm{ times}})$ to $(6;\underbrace{\mathbf{o_{m}},\ldots,\mathbf{o_{m}}}_{n\textrm{ times}})$. Let us denote by $(i;{\bf o_1},\ldots,{\bf o_n})\&\overline{{\bf 1}}$ the concatenated walk $(i;{\bf o_1},\ldots,{\bf o_n},\overline{{\bf 1}})$. Then $\partial\T_n$ is obtained by joining by straight line segments the points
$$\psi(w_1^{(n)}\&\overline{\mathbf{1}}),\ldots,\psi(w_{m_n}^{(n)}\&\overline{\mathbf{1}}),\psi(w_1^{(n)}\&\overline{\mathbf{1}})
$$
in this order. Each vertex of $\partial\T_n$ corresponds  to an infinite walk ending up in a cycle of $G$. Thus these are images of fixed points of contractions:
$$(f_{a_1}\circ\ldots\circ f_{a_l})\left(\textrm{Fix}(f_{a_{l+1}}\circ\ldots \circ f_{a_{l+n}})\right),
$$ 
where $f_a(x):=M^{-1}(x+a)$ for each $a\in \mathcal{D}$.
\end{remark}

The first terms of the approximation sequence are depicted in Figures~\ref{KnuthApprox} to~\ref{ABApprox} for some examples. 
\begin{figure}
\begin{center}
\begin{tabular}{lllll}
\includegraphics[width=30mm,height=20mm]{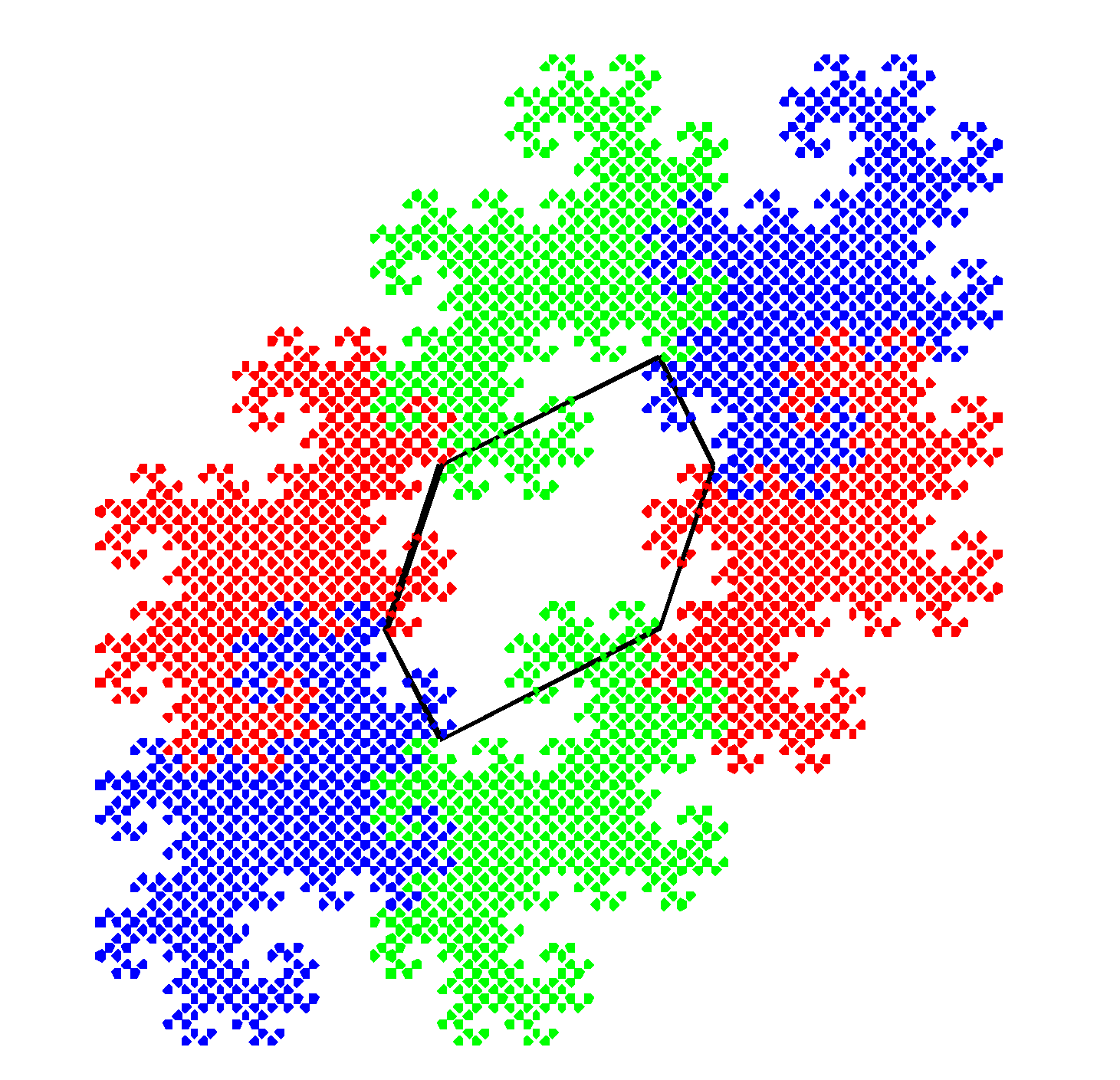}
&
\includegraphics[width=30mm,height=20mm]{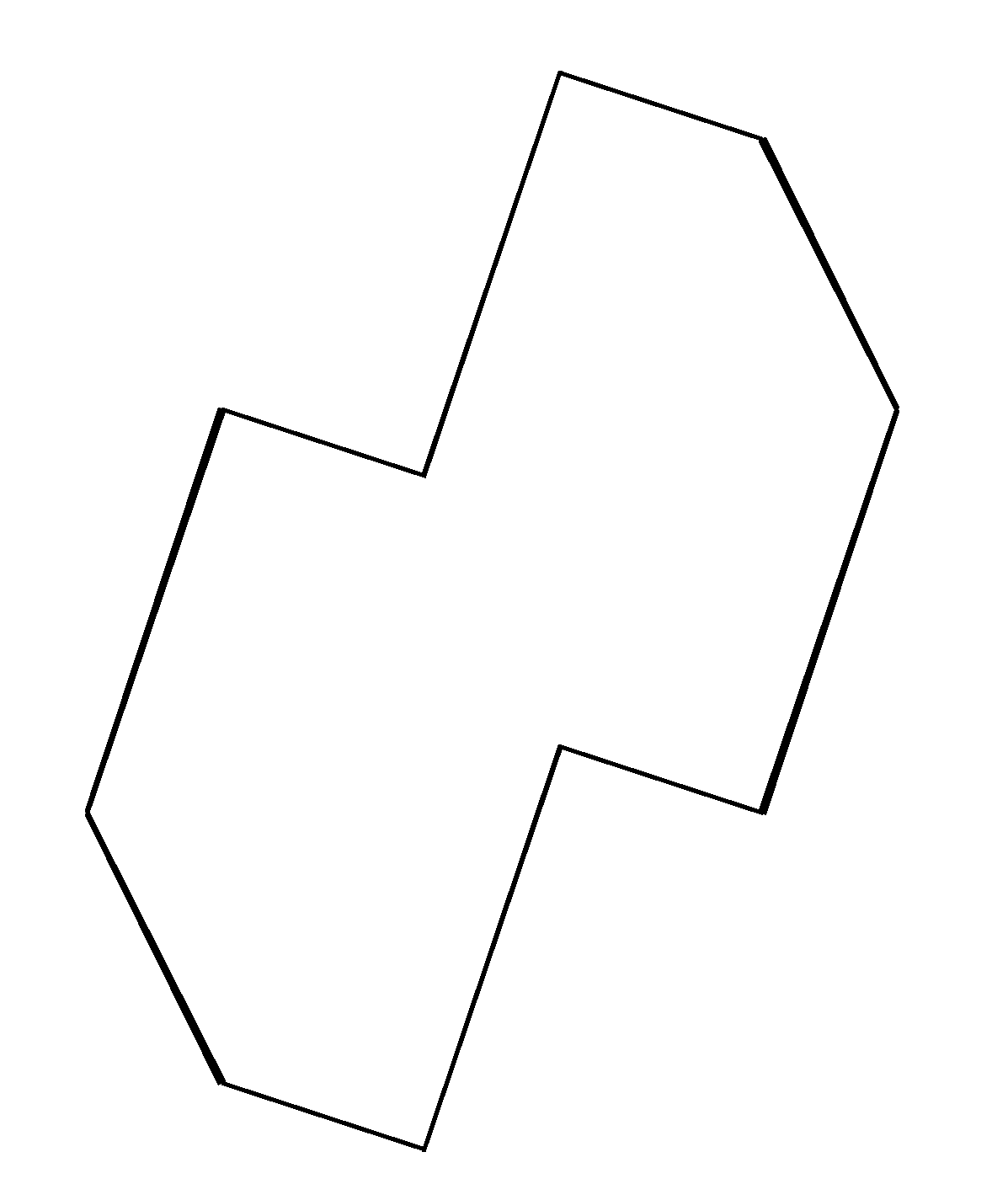}
&
\includegraphics[width=30mm,height=20mm]{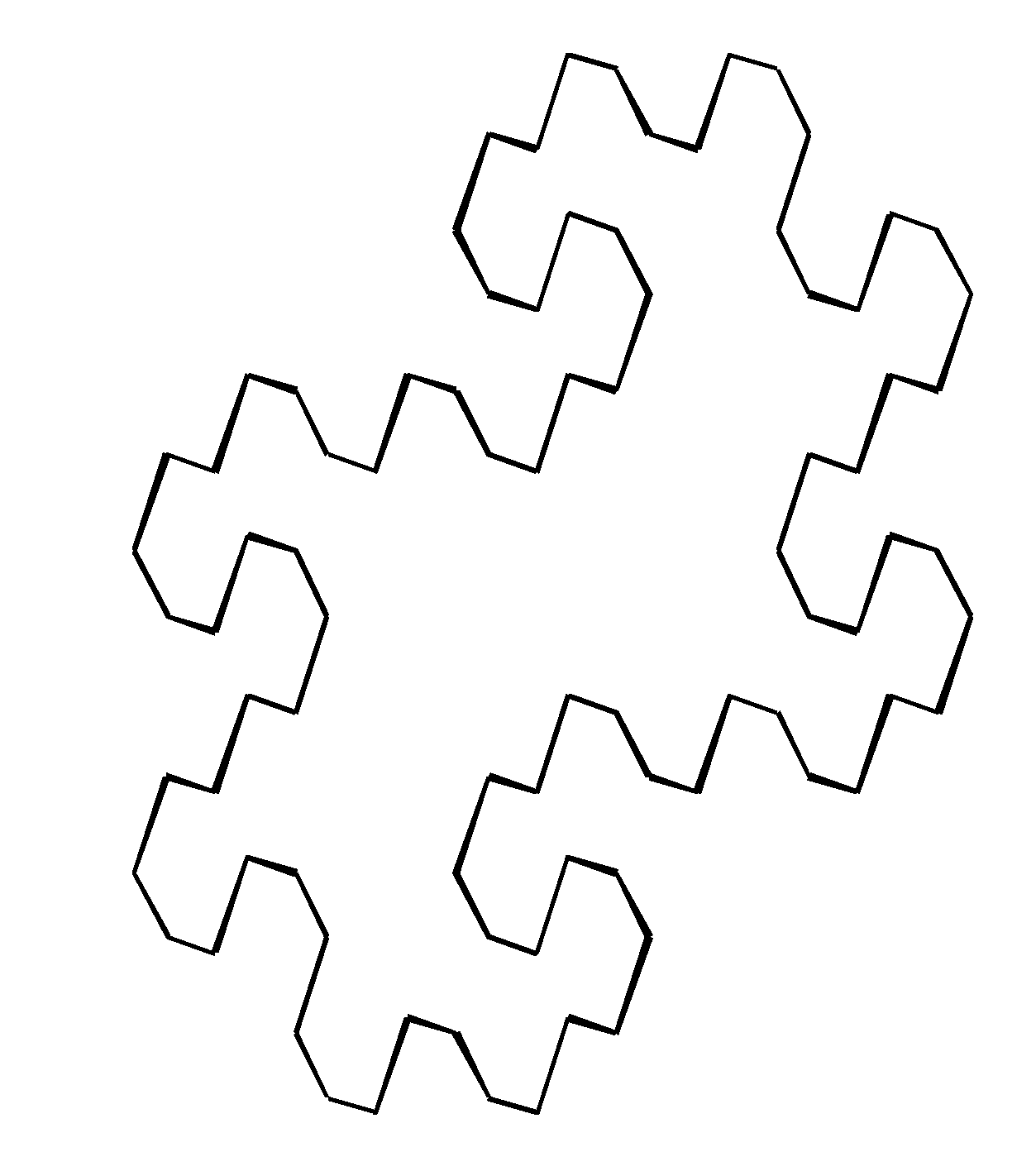}
&
\includegraphics[width=30mm,height=20mm]{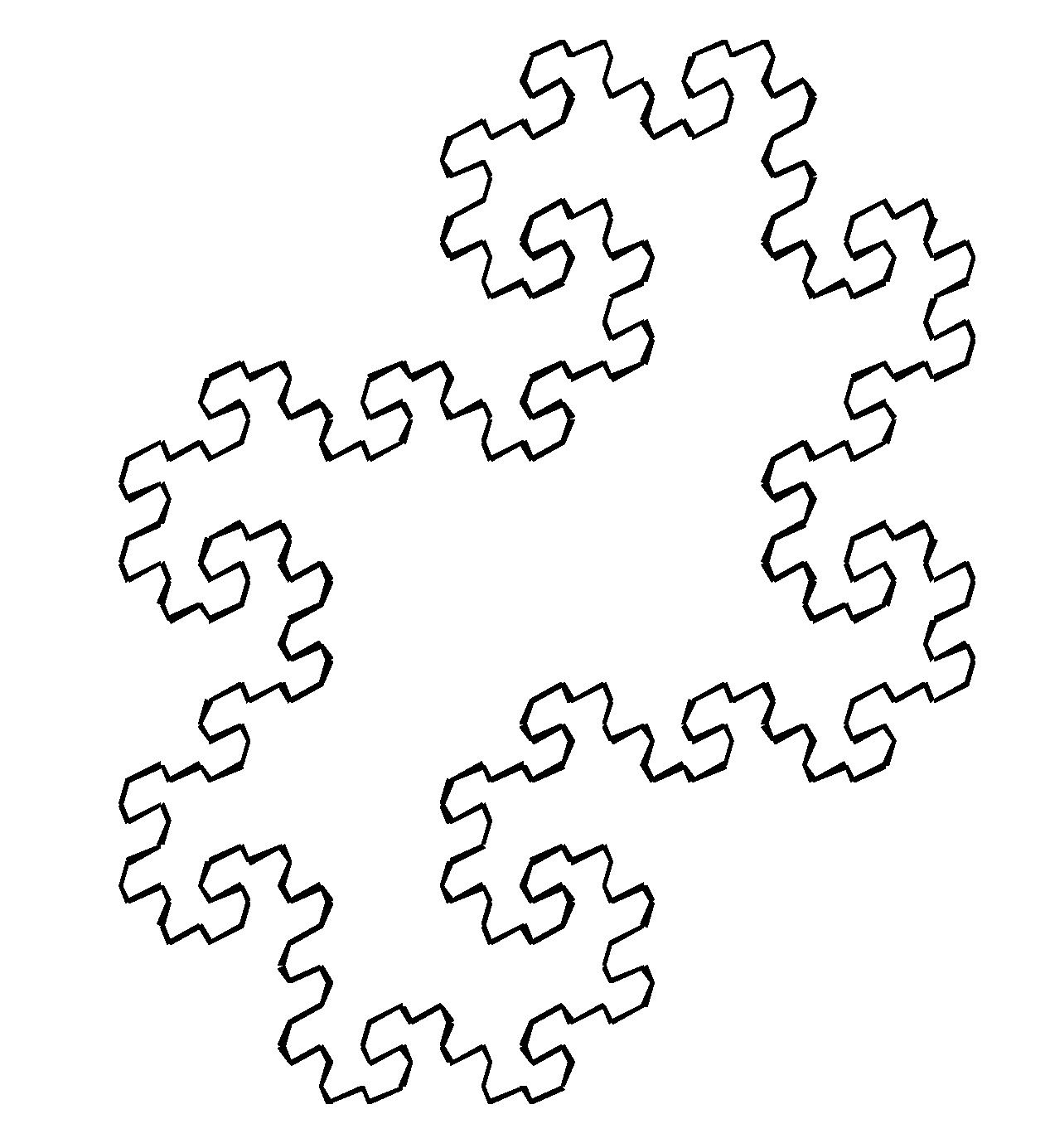}
&
\includegraphics[width=30mm,height=20mm]{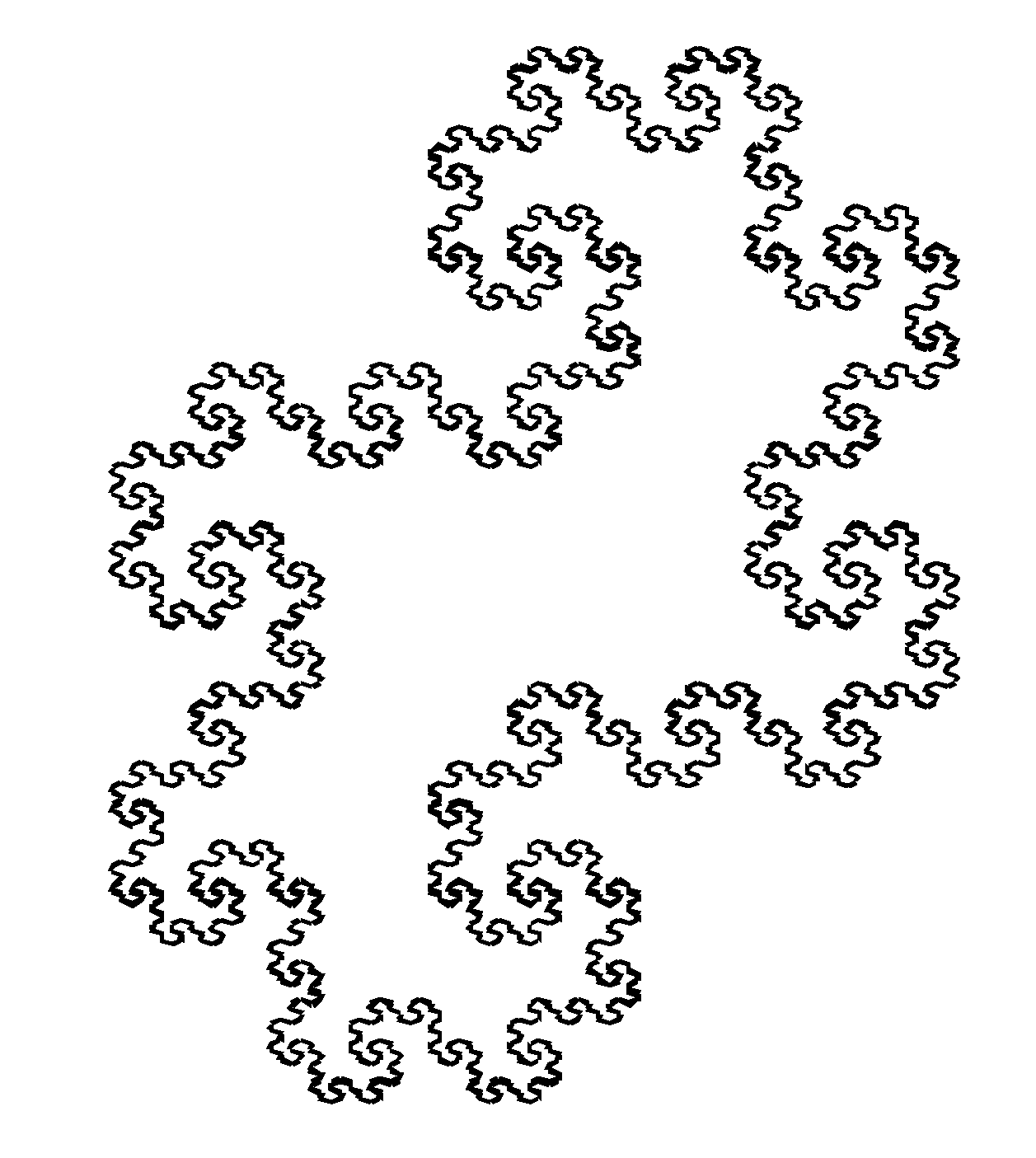}
\end{tabular}
\end{center}
\caption{Knuth dragon ($A=B=2)$ : tiling with $\partial\T_0$; $\partial\T_1$,$\partial\T_5,\partial\T_8,\partial\T_{11}$}\label{KnuthApprox}
\end{figure}

\begin{figure}
\begin{center}
\begin{tabular}{ccccc}
\includegraphics[width=30mm,height=20mm]{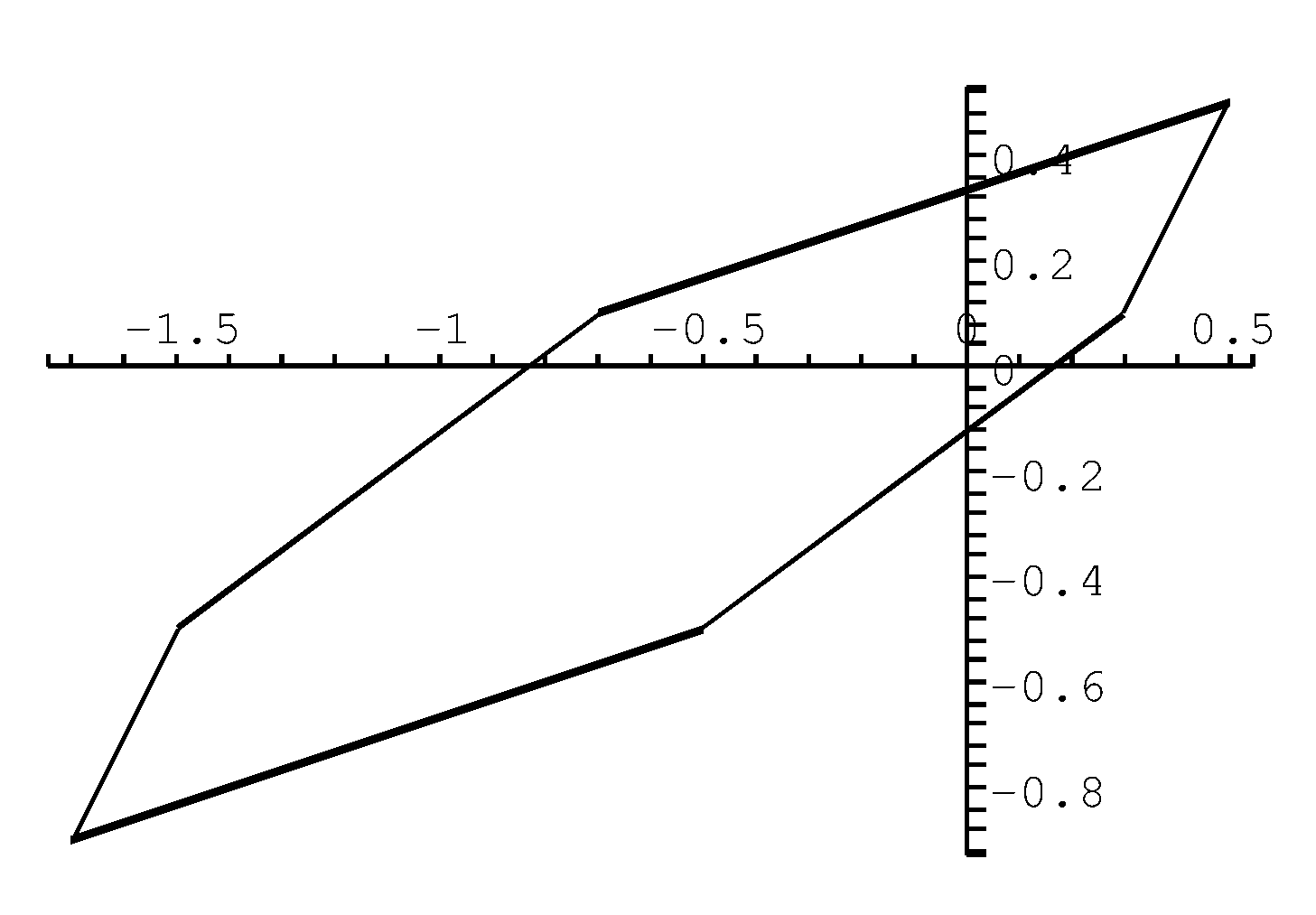}
&
\includegraphics[width=30mm,height=20mm]{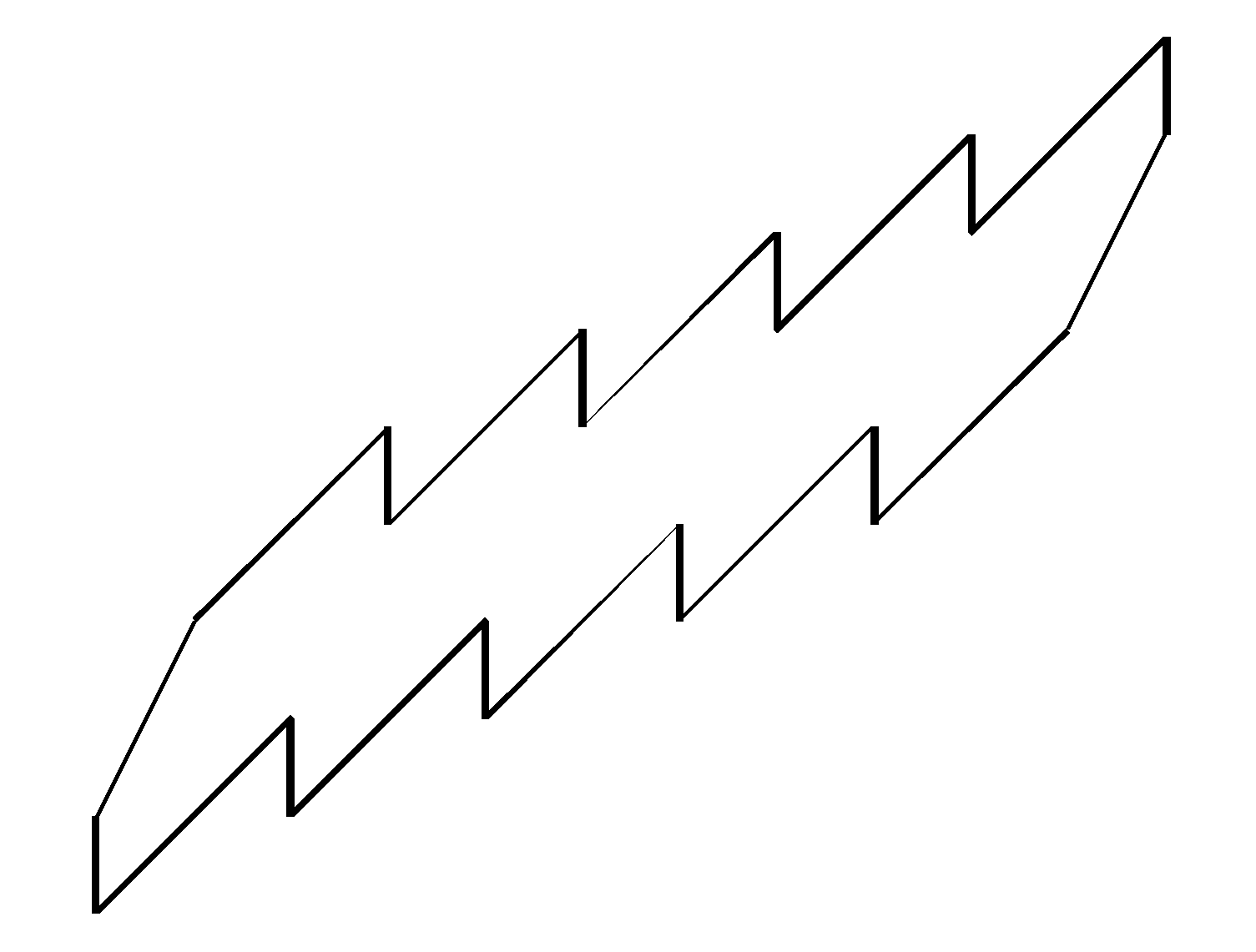}
&
\includegraphics[width=30mm,height=20mm]{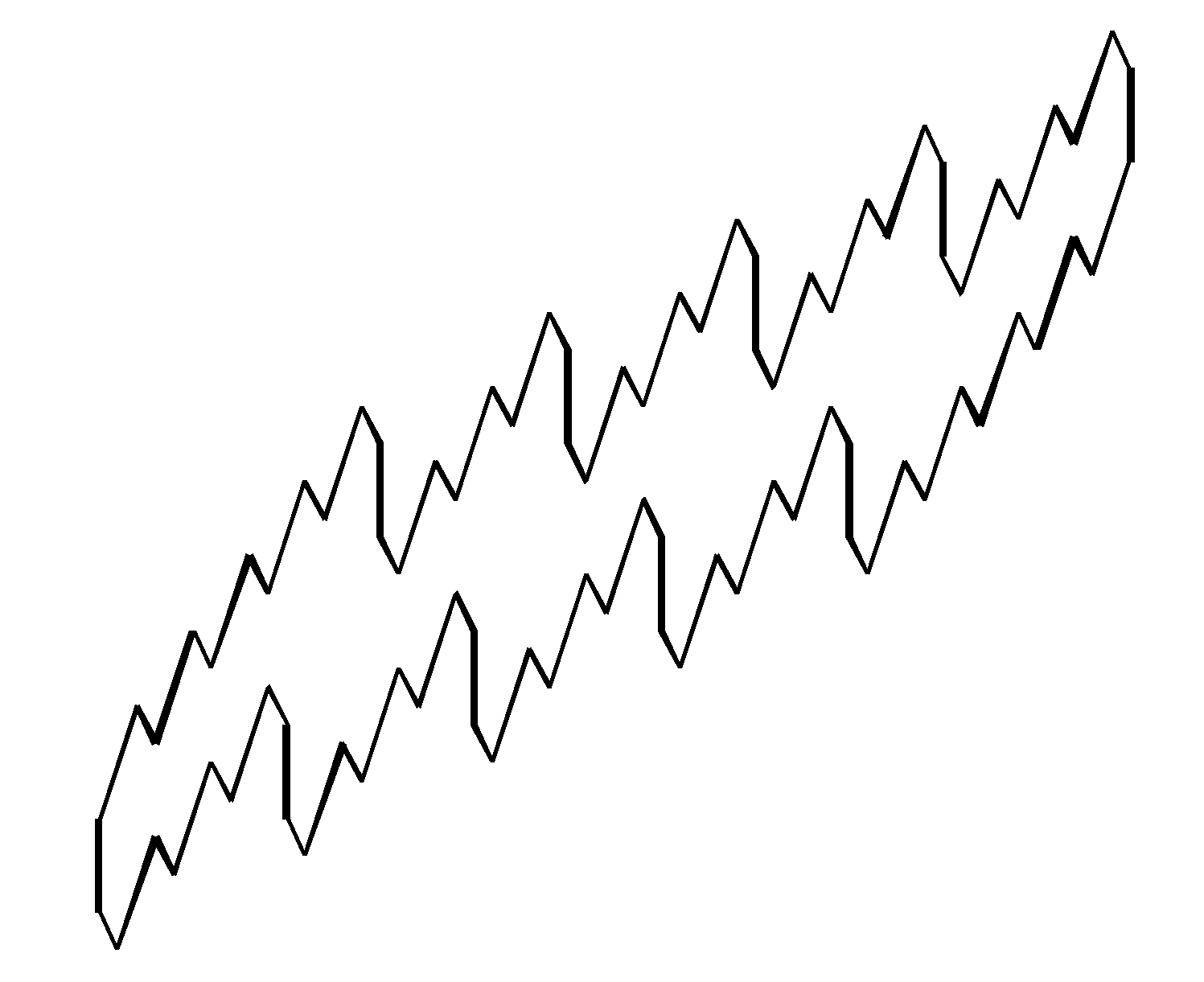}
&
\includegraphics[width=30mm,height=20mm]{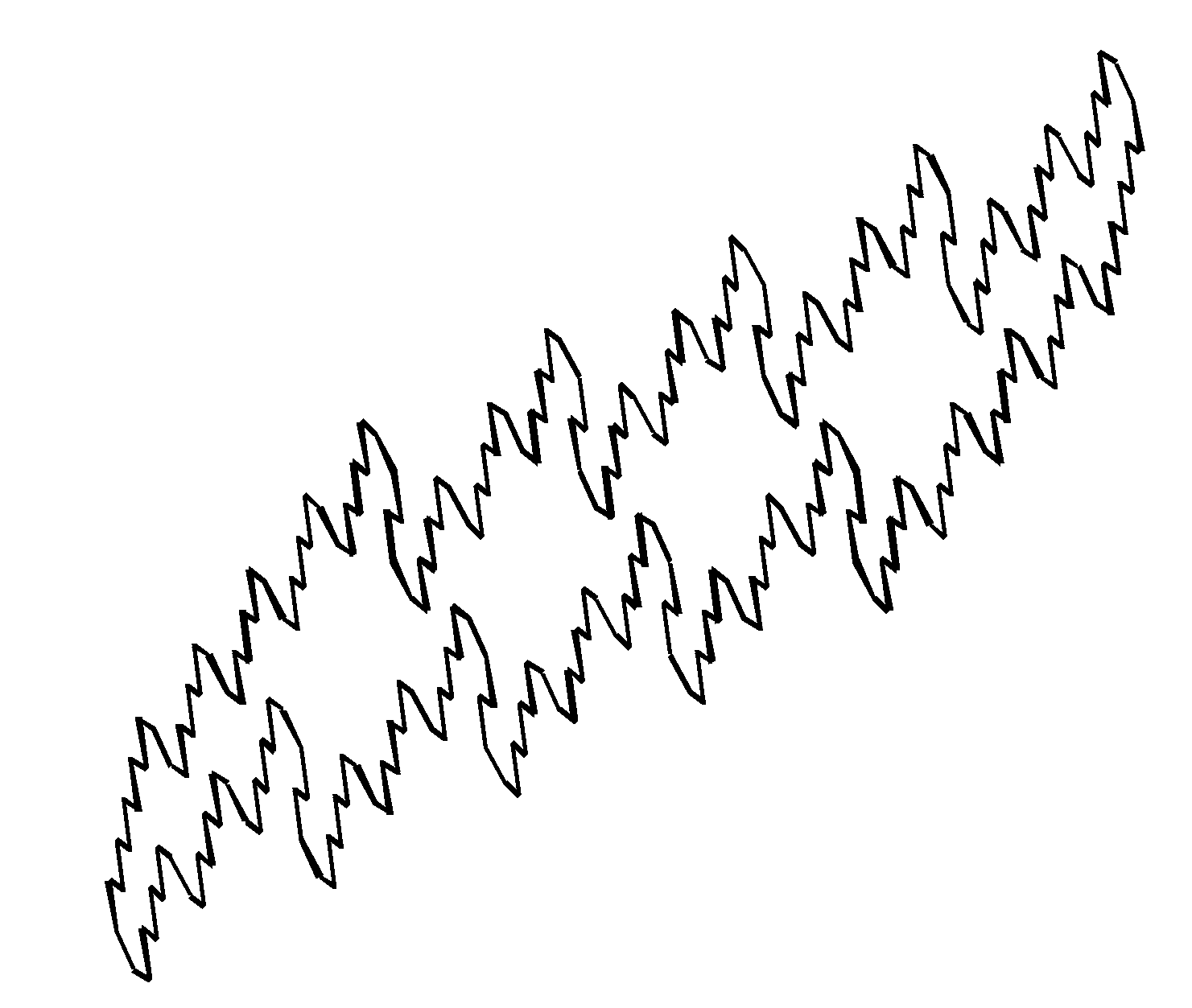}
&
\includegraphics[width=30mm,height=20mm]{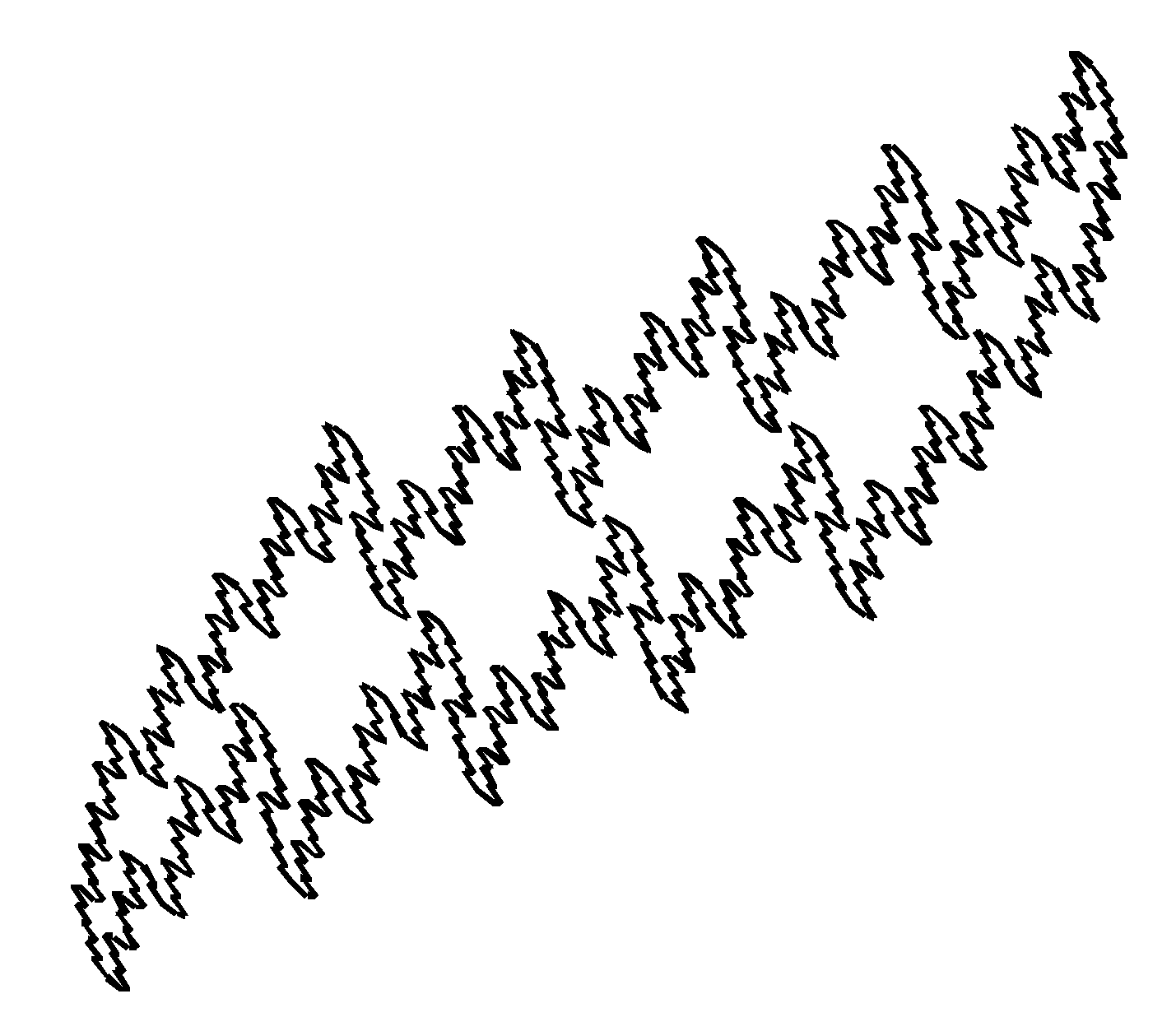}
\end{tabular}
\end{center}
\caption{$A=4,B=5$ : $\partial\T_0,\ldots,\partial\T_4$}\label{A4B5Approx}
\end{figure}

\begin{figure}
\begin{center}
\begin{tabular}{cccc}
\includegraphics[width=35mm,height=23mm]{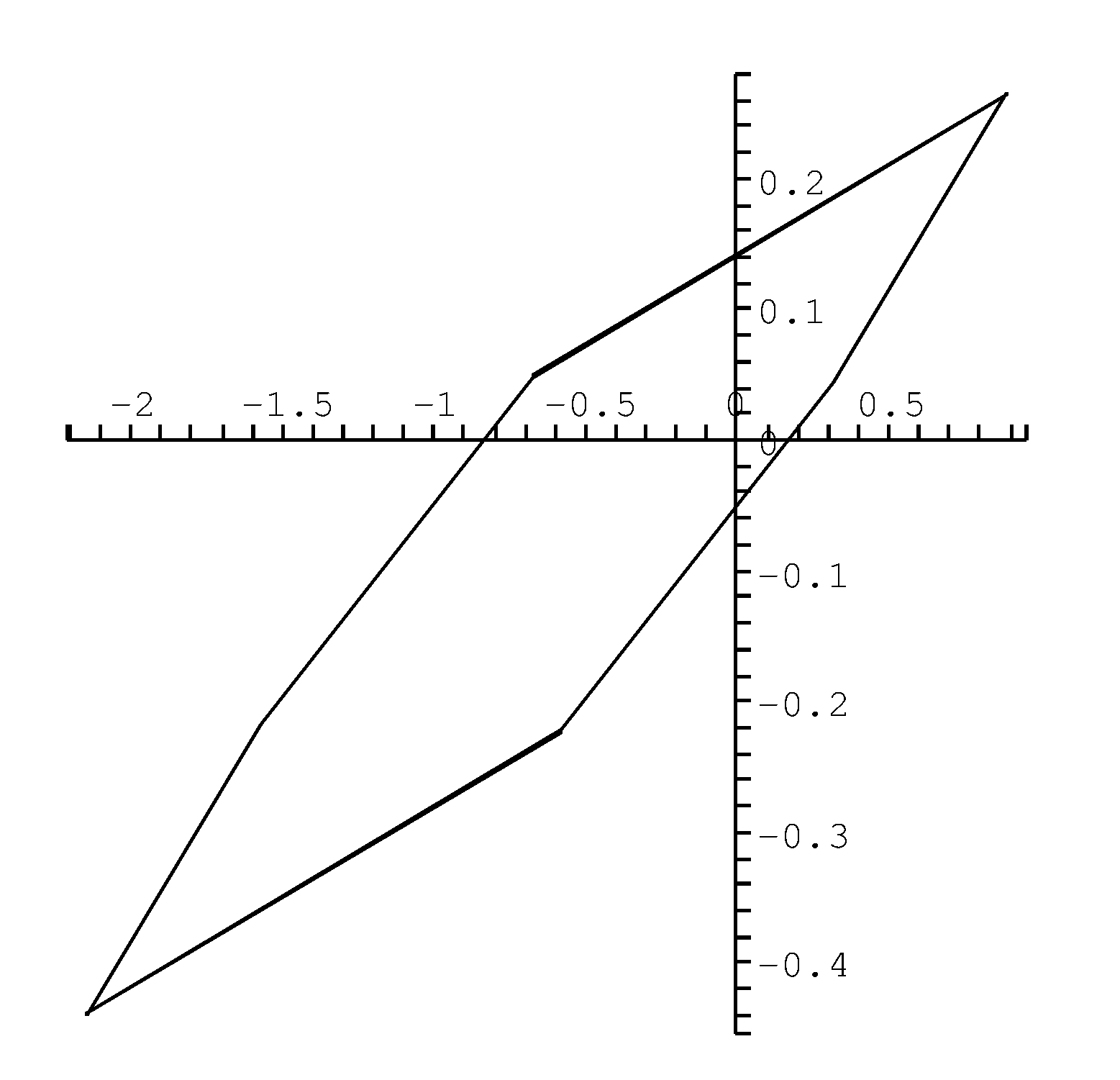}
&
\includegraphics[width=35mm,height=23mm]{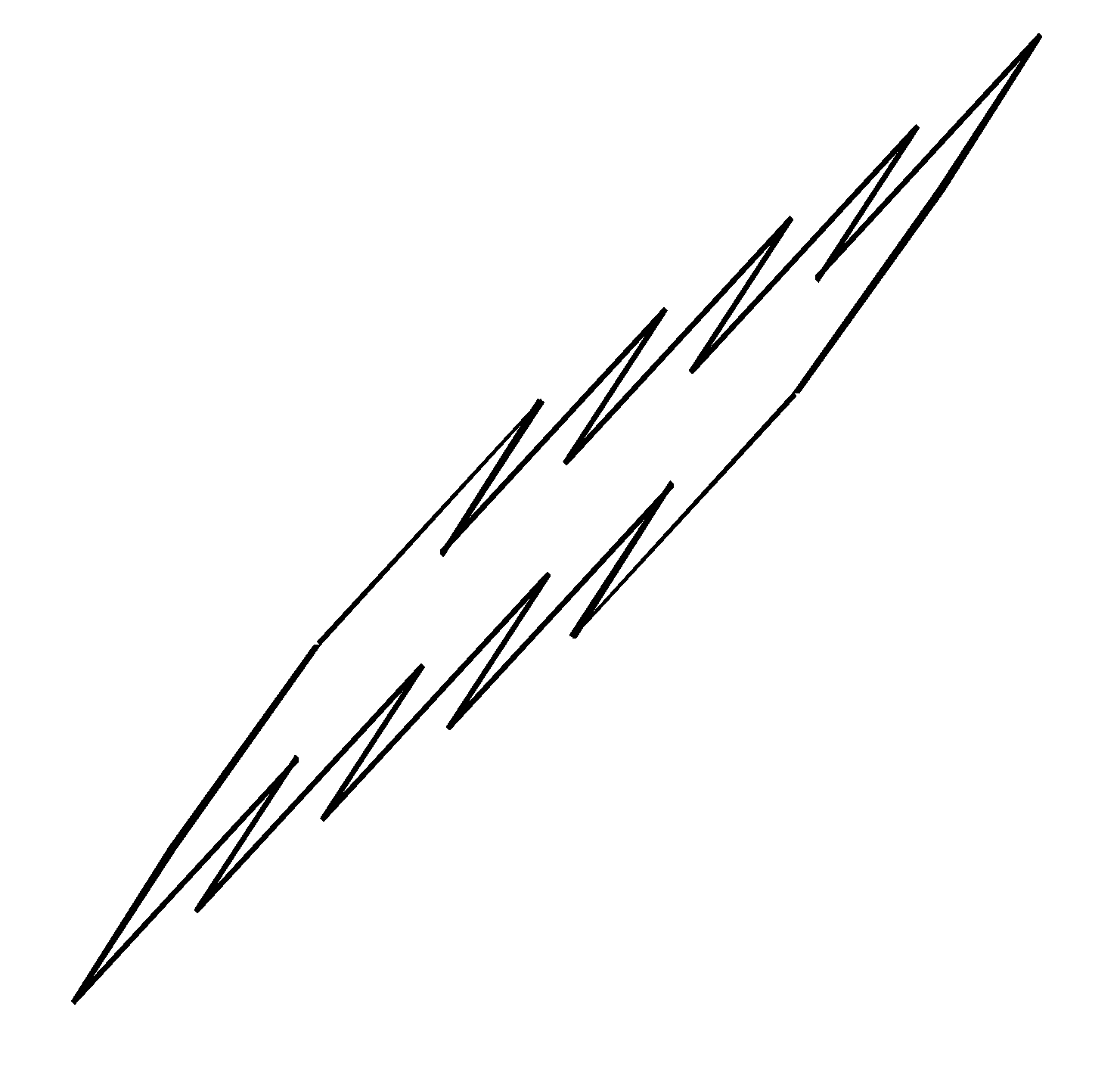}
&
\includegraphics[width=35mm,height=23mm]{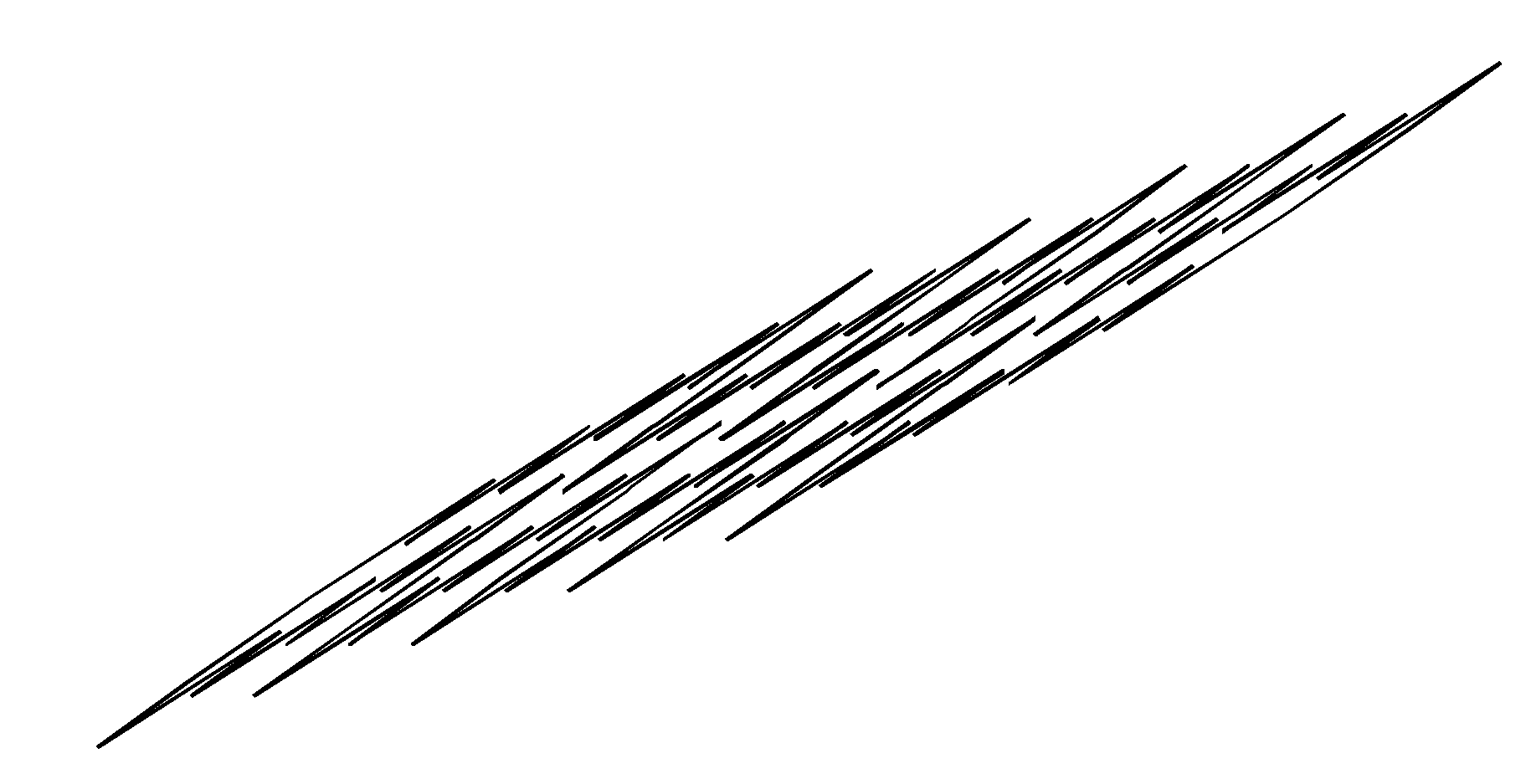}
&
\includegraphics[width=35mm,height=23mm]{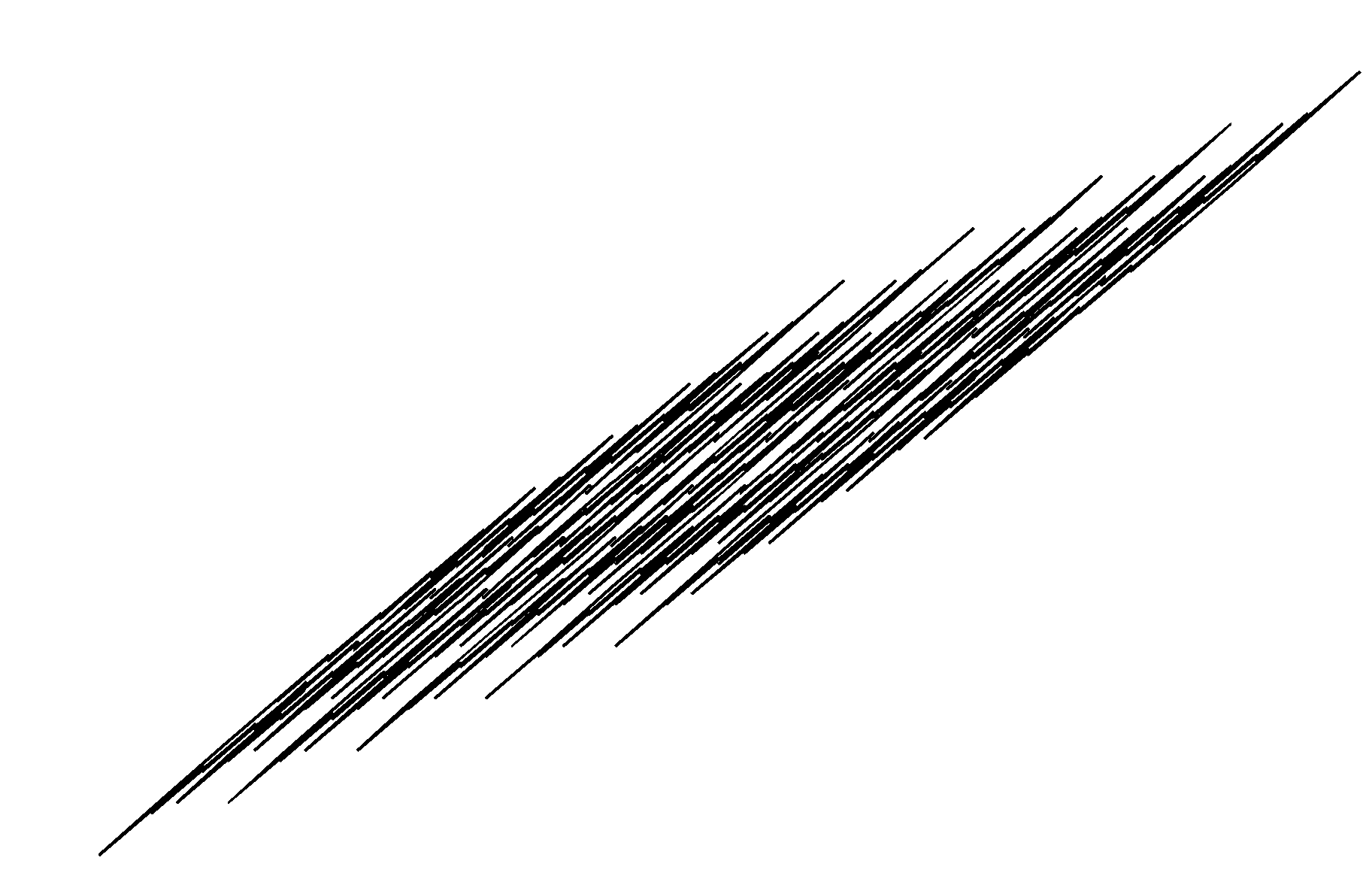}
\end{tabular}
\end{center}
\caption{$A=5,B=5$ : $\partial\T_0,\ldots,\partial\T_3$}\label{ABApprox}
\end{figure}

For the construction of the curve $Q\subset\partial\T$ meeting the assumptions of Theorem~\ref{th:nt04}, we will make use of neighbor relations of the natural subdivisions of $\T$. This will require the knowledge of the whole neighbor set of $\T$ itself. Note that, for $2A-B\in\{3,4\}$, the quantity $J$ of Equation~(\ref{J}) is equal to $2$ (or $J=3$ in the case $A=B=4$, but then $x^2+4x+4=(x+2)^2$ and $\T$ is a square). Therefore, the neighbor set (see Equation~(\ref{eq:neighset})) reduces to 
\begin{equation}\label{eq:neighset34}
\mathcal{S}=\left\{\pm\left(\begin{array}{c}1\\0\end{array}\right),
\pm\left(\begin{array}{c}A\\1\end{array}\right),
\pm\left(\begin{array}{c}A-1\\1\end{array}\right),
\pm\left(\begin{array}{c}A-2\\1\end{array}\right),
\pm\left(\begin{array}{c}2A-2\\2\end{array}\right)
\right\}.
\end{equation}

\end{section}

\begin{section}{$\T$ has no cut point for $2A-B=3$}\label{sec:thnocp3} 
We assume throughout this section that $2A-B=3$ and $A\ne B$. The case $A=B=3$ can be treated in a similar way, with a simplified graph $G$ as mentioned in Remark~\ref{rem:AeqB}. The aim is to construct a curve $Q\subset \partial \T$ satisfying the assumptions of Theorem~\ref{th:nt04}. We will use the following terminology.
\begin{definition}
Let $M_1,M_2,\ldots ,M_{n},n\geq 3,$ be compact sets in
the plane. If $\#\left(M_i\cap M_{i+1}\right)=1$ for $1 \leq i
\leq n-1$ and $M_i\cap M_j=\emptyset $ for $\left| i-j\right| \geq
2$, we say that $M_1,M_2,\ldots, M_n$ form a {\em chain}. If
$\#\left(M_i\cap M_{i+1}\right)=1$ for $1 \leq i \leq n-1$,
$\#\left(M_{n}\cap M_1\right)=1$ and $M_i\cap M_j=\emptyset $ for
$2\leq \left| i-j\right| \leq n-2$, we say that $M_1,M_2,\ldots
,M_{n}$ form a {\em circular chain}. 
\end{definition}

\subsection{Construction of $Q$.}\label{sub:constrQ}
The curve $Q\subset\partial \T$ is built up from simple arcs according to the following strategy. 
\begin{itemize}
\item[$1$.] We define a chain of curves $\alpha_1,\alpha_2,\ldots, \alpha_B$, along the bottom of the boundary of $\T$, as images by the boundary parametrization $C$ of subintervalls of $[0,1]$. More precisely, for $i\in\{1,\ldots B\}$, let
$$\alpha_i:=C([s_i,t_i]),
$$
where $s_i,t_i\in[0,1]$ are the parameters associated to the infinite walks of $G^o$ given in Table~\ref{tab:alpha}. We check that $\alpha_1,\alpha_2,\ldots, \alpha_B$ form a chain in Section~\ref{sub:alphachain}.
\item[$2$.] We denote by $\alpha_i'$, $i\in\{1,\ldots B\}$,  the curve along the top of the boundary of $\T$, obtained by flipping the curve $\alpha_i$ through the exchange of digits 
$a\leftrightarrow B-1-a\;\;(a\in\{0,\ldots,B-1\}).
$ 
In other words,
$$\alpha_i':=\left\{\sum_{j=1}^\infty M^{-j}\left(\begin{array}{c}B-1-a_j\\0\end{array}\right);\sum_{j=1}^\infty M^{-j}\left(\begin{array}{c}a_j\\0\end{array}\right)\in\alpha_i\right\}.
$$
In particular, $\alpha_i'=0.\overline{B-1}-\alpha_i$. Moreover, since the boundary language, given by $G$, is invariant by this flipping, the curves $\alpha_i'$ are also subsets of $\partial\T$. Indeed, it can be seen on the left of Figure~\ref{CNS_DTautomaton} that $s\xrightarrow{a}s'$ $\iff$ $-s\xrightarrow{B-1-a}-s'$. We will check in Section~\ref{sub:circchain} that $\alpha_1,\ldots,\alpha_B,\alpha_1',\ldots,\alpha_B'$ is a circular chain.
\item[$3$.] For each $i\in\{1,\ldots,B\}$, $\alpha_i$ is a  locally connected continuum of the Euclidean plan. In particular, it is arcwise connected and there exists a simple arc $\beta_i\subset \alpha_i$ such that  $\beta_i$ has the same endpoints as $\alpha_i$. Similarly to Item $2.$, let $\beta_i'=0.\overline{B-1}-\beta_i$. Then $\beta_i'\subset\alpha_i'$ is a simple arc having the same endpoints as $\alpha_i'$. Therefore, $\beta_1,\ldots,\beta_B,\beta_1',\ldots,\beta_B'$ is a circular chain of simple arcs joined by their endpoints, hence their union 
$$Q':=\bigcup_{i=1}^B\beta_i\cup \bigcup_{i=1}^B\beta_i'\subset\partial \T
$$ is a simple closed curve. 
\item[$4$.] We add further simple arcs $\gamma_i$, $i\in\{1,\ldots,B-2\}$, whose endpoints lie on $Q'$. More precisely, 
$$\gamma_i:=f_if_{B-1}^{-1}\left(\beta_{B-1}\cup\beta_{B}\right)
$$
is a simple arc on $\partial \T$ going through the point $0.i\overline{A-2}$ and with endpoints on $Q'$ (see Section~\ref{sub:gamma}).
\item[$5$.] The curve 
$$Q:=Q'\cup\bigcup_{i=1}^{B-2}\gamma_i
$$
meets the assumptions of Theorem~\ref{th:nt04}, as proved in Section~\ref{sub:checkassu}.
\end{itemize}
Note that this strategy generalizes to a large range of parameters $A,B$ the construction of $Q$ obtained for the case $A=4,B=5$ in~\cite{NgaiTang04}, with the additional property that $Q$ is now a subset of $\partial\T$.\\

\begin{table}
{\tiny
$$\begin{array}{|c|c|c|c|c|}
 \hline
&s_i&t_i&C(s_i)&C(t_i)\\\hline
\rule{0pt}{3ex}\rule[-1.6ex]{0pt}{0pt} \alpha_{1}&(6;\mathbf{2(B-A)-2},\mathbf{1},\mathbf{B-2},\overline{\mathbf{2}})
&(6;\mathbf{2(B-A)-1},\mathbf{2A-2},\mathbf{4},\overline{\mathbf{2}})
& .10(B-1)\overline{(B-1)0}&.1(A-2)(B-2)\overline{0(B-1)}
\\\hline
\vdots&\vdots&\vdots&\vdots&\vdots\\\hline
\rule{0pt}{3ex}\rule[-1.6ex]{0pt}{0pt} \alpha_{B-A-1}&(6;\mathbf{2},\mathbf{1},\mathbf{B-2},\overline{\mathbf{2}})
&(6;\mathbf{3},\mathbf{2A-2},\mathbf{4},\overline{\mathbf{2}})
&.(B-A-1)0(B-1)\overline{(B-1)0}&.(B-A-1)(A-2)(B-2)\overline{0(B-1)}
\\\hline
 \rule{0pt}{3ex}\alpha_{B-A}&(5;\mathbf{2A-1},\mathbf{1},\mathbf{B-2},\overline{\mathbf{2}})
&(6;\mathbf{1},\mathbf{2A-2},\mathbf{4},\overline{\mathbf{2}})
&.(B-A)0(B-1)\overline{(B-1)0}&.(B-A)(A-2)(B-2)\overline{0(B-1)}
\\\hline
\rule{0pt}{3ex}\rule[-1.6ex]{0pt}{0pt} \alpha_{B-A+1}&(5;\mathbf{2A-3},\mathbf{1},\mathbf{B-2},\overline{\mathbf{2}})
&(5;\mathbf{2A-2},\mathbf{2A-2},\mathbf{4},\overline{\mathbf{2}})&&
\\\hline
\vdots&\vdots&\vdots&\vdots&\vdots\\\hline
\rule{0pt}{3ex}\rule[-1.6ex]{0pt}{0pt} \alpha_{B-2}&(5;\mathbf{3},\mathbf{1},\mathbf{B-2},\overline{\mathbf{2}})
&(5;\mathbf{4},\mathbf{2A-2},\mathbf{4},\overline{\mathbf{2}})
&.(B-2)0(B-1)\overline{(B-1)0}&
\\\hline
\rule{0pt}{3ex}\rule[-1.6ex]{0pt}{0pt} \alpha_{B-1}&(5;\mathbf{2},\overline{\mathbf{2A-2}})
&(5;\mathbf{2},\mathbf{2A-2},\mathbf{4},\overline{\mathbf{2}})
&.(B-1)\overline{A-2}&.(B-1)(A-2)(B-2)\overline{0(B-1)}
\\\hline
\rule{0pt}{3ex}\rule[-1.6ex]{0pt}{0pt} \alpha_{B}&(3;\mathbf{2(B-A)},\mathbf{1},\mathbf{2(B-A)+1},\overline{\mathbf{2}})
&(3;\mathbf{2(B-A)+1},\overline{\mathbf{2A-2}})
&.(B-1)(B-1)0\overline{0(B-1)}&.(B-1)\overline{A-2}
\\\hline
\end{array}$$
}
\caption{Endpoints of the curves $\alpha_i$, $i\in\{1,\ldots,B\}$. Case $2A-B=3,\;A\ne B$.}\label{tab:alpha}
\end{table}

\subsection{The curves $\alpha_1,\alpha_2,\ldots,\alpha_B$ form a chain}\label{sub:alphachain}

The proof of this fact relies on a lemma giving an account of the intersections of the natural subdivisions of $\T$ obtained by iterating (\ref{SATile}). For $a_1,\ldots, a_m\in\{0,\ldots, B-1\}$, we denote by $\T_{a_1\ldots a_m}$ the subdivision $M^{-1}\left(\begin{array}{c}a_1\\0\end{array}\right)+\cdots+M^{-m}\left(\begin{array}{c}a_m\\0\end{array}\right)+M^{-m}\T$. We also further use the notations~\eqref{eq:not} and~\eqref{eq:not2}.

\begin{lemma}\label{lem:subd} The following assertions hold for digits $a_1,a_1',\ldots,a_4,a_4'\in\{0,\ldots,B-1\}$ with $a_1\ne a_1'$.
\begin{itemize}
\item $\T_{a_1}\cap\T_{a_1'}\ne\emptyset\iff a_1-a_1'=\pm1$.
\item Suppose $a_1-a_1'=1$. Then 
$$\T_{a_1a_2}\cap\T_{a_1'a_2'}\ne\emptyset\iff a_2-a_2'\in\{A,A-1,A-2\}.
$$
\item Suppose $a_1-a_1'=1$ and $a_2-a_2'=A$. Then 
$$\T_{a_1a_2a_3a_4}\cap\T_{a_1'a_2'a_3'a_4'}\ne\emptyset\iff 
\left\{\begin{array}{l}a_3=B-1,a_3'=0\\a_4-a_4'\in\{-A,-A+1,-A+2\}.\end{array}\right.
$$
\item Suppose $a_1-a_1'=1$ and $a_2-a_2'=A-1$. Then 
$$\begin{array}{cl}&\T_{a_1a_2a_3a_4}\cap\T_{a_1'a_2'a_3'a_4'}\ne\emptyset\\\\
\iff &\left\{\begin{array}{l}
a_3-a_3'=B-A\\a_4=0,a_4'=B-1
\end{array}\right.
\textrm{or } \left\{\begin{array}{l} a_3-a_3'=B-A+1\\a_4-a_4'\in\{A-B,A-B-1,A-B-2\}\end{array}\right.\\\\
&\textrm{or }\left\{\begin{array}{l} a_3-a_3'=B-A+2\\a_4-a_4'=1.
\end{array}\right.
\end{array}
$$
\item Suppose $a_1-a_1'=1$ and $a_2-a_2'=A-2$. Then 
$$\T_{a_1a_2a_3}\cap\T_{a_1'a_2'a_3'}\ne\emptyset\iff 
a_3-a_3'=-1.
$$
In this case, 
$$\T_{a_1a_2a_3}\cap\T_{a_1'a_2'a_3'}=\{0.a_1a_2a_3\overline{0(B-1)}\}=\{0.a_1'a_2'a_3'\overline{(B-1)0}\}.
$$
\end{itemize}
\end{lemma}
\begin{proof} The following equivalences hold. 
\begin{itemize}
\item $\T_{a_1}\cap\T_{a_1'}\ne\emptyset
\iff M\left(\T_{a_1}\cap\T_{a_1'}\right)\ne\emptyset
\iff \left(\begin{array}{c}a_1-a_1'\\0\end{array}\right)\in\mathcal{S}$.
\item[]
\item 
$\T_{a_1a_2}\cap\T_{a_1'a_2'}\ne\emptyset
\iff M^2\left(\T_{a_1a_2}\cap\T_{a_1'a_2'}\right)\ne\emptyset
\iff\left(\begin{array}{c}a_2-a_2'\\a_1-a_1'\end{array}\right)
\in\mathcal{S}$.
\item[]
\item Using the fact that $M^2=-AM-B I_2$, we further have 
$$\begin{array}{c}
\T_{a_1a_2a_3}\cap\T_{a_1'a_2'a_3'}\ne\emptyset\iff M^3\left(\T_{a_1a_2a_3}\cap\T_{a_1'a_2'a_3'}\right)\ne\emptyset\\\\
\iff

\left(\begin{array}{c}a_3-a_3'-B(a_1-a_1')\\a_2-a_2'-A(a_1-a_1')\end{array}\right)\in\mathcal{S}.\end{array}$$
\item[]
\item Finally,
$$\begin{array}{c}
\T_{a_1a_2a_3a_4}\cap\T_{a_1'a_2'a_3'a_4'}\ne\emptyset
\iff M^4\left(\T_{a_1a_2a_3a_4}\cap\T_{a_1'a_2'a_3'a_4'}\right)\ne\emptyset\\\\
\iff
\left(\begin{array}{c}a_4-a_4'-B(a_2-a_2')+AB(a_1-a_1')\\a_3-a_3'-A(a_2-a_2')+(A^2-B)(a_1-a_1')\end{array}\right)
\in\mathcal{S}.
\end{array}
$$
\end{itemize}
 Now, $\mathcal{S}$ is given in~(\ref{eq:neighset34}) and this proves directly the first four assertions of the lemma, as well as the last assertion by induction.
\end{proof} 
\begin{remark}\label{rem:eqpoints}In particular, it follows from the last item of the above lemma that, for digits  $a_1,a_1',a_2,a_2',a_3,a_3'\in\mathcal{D}$, we have
$$\left.\begin{array}{rcl}a_1-a_1'&=&1\\ a_2-a_2'&=&A-2\\a_3-a_3'&=&-1\end{array}\right\}\;\Rightarrow \;0.a_1a_2a_3\overline{0(B-1)}=0.a_1'a_2'a_3'\overline{(B-1)0}.
$$
This will be used several times in order to check that the endpoints of certain boundary parts of $\T$ coincide.
\end{remark}

The next lemma gives a precise description of the points constituting the curves $\alpha_1,\ldots,\alpha_B$. It will be used to prove that certain boundary parts have either empty or one-point intersection. According to the ordering chosen for the states of the graph $G$, we denote by 
$$K_1:=K_{-R},K_2:=K_{Q_1},K_3:=K_{-P_1},K_4:=K_{R},K_5:=K_{-Q_1},K_6:=K_{P_1}
$$ 
the compact sets of Proposition~\ref{prop:primsubautomaton}. Moreover, for a finite string of digits $a_1\ldots a_m$ and a set $K\in\{K_1,\ldots,K_6\}$,  we write
$$0.a_1\ldots a_m[K]:=\left\{0.a_1\ldots a_mb_1b_2\ldots;0.b_1b_2\ldots\in K\right\}.
$$
\begin{lemma}\label{lem:descalpha}
For $i=1,\ldots,B-2$, we have
$$\begin{array}{rcl}
\alpha_i &=&\{0.i0(B-1)\overline{(B-1)0}\}\\\\
&&\cup\;\bigcup_{p\geq0}0.i0(B-1)((B-1)0)^p (B-A)[K_1] \;\cup\; \bigcup_{p\geq0}\bigcup_{k=B-A+1}^{B-2}0.i0(B-1)((B-1)0)^p k[K_1\cup K_2]
\\\\
&&\cup\; \bigcup_{p\geq0}\bigcup_{k=1}^{A-2}0.i0(B-1)(B-1)(0(B-1))^p k[K_4\cup K_5] \\\\
&&\hspace{1cm}\cup\;  \bigcup_{p\geq0}0.i0(B-1)(B-1)(0(B-1))^p (A-1)[K_4] 
\\\\
&&\cup\;\bigcup_{k=0}^{A-3}0.ik[K_4\cup K_5]\;\cup\;0.i(A-2)[K_4]\\\\
&&\cup\;0.i(A-2)(B-2)[K_1]\;\cup\;0.i(A-2)(B-1)[K_1\cup K_2] \\\\
&&\cup\;\bigcup_{p\geq0}0.i(A-2)(B-2)(0(B-1))^p0[K_4]
\;\cup\;\bigcup_{p\geq0}0.i(A-2)(B-2)0((B-1)0)^p(B-1)[K_1]\\\\
&&\cup\;\{0.i(A-2)(B-2)\overline{0(B-1)}\}.
\end{array}
$$
Moreover, 
$$\begin{array}{rcl}
\alpha_{B-1} &=&\{0.(B-1)(A-2)(B-2)\overline{0(B-1)}\}\\\\
&&\cup\; \bigcup_{p\geq0}\bigcup_{k=0}^{A-2}0.(B-1)(A-2)(B-2)(0(B-1))^p k[K_4\cup K_5]\\\\
&&\hspace{1cm}\cup\;\bigcup_{p\geq0}0.(B-1)(A-2)(B-2)(0(B-1))^p (A-1)[K_4]
\\\\
&&\cup\; \bigcup_{p\geq0}0.(B-1)(A-2)(B-2)0((B-1)0)^p (B-A)[K_1]  \\\\
&&\hspace{1cm}\cup\;   \bigcup_{p\geq0}\bigcup_{k=B-A+1}^{B-1}0.(B-1)(A-2)(B-2)0((B-1)0)^p k[K_1\cup K_2]
\\\\
&&\cup\;0.(B-1)(A-2)(B-A+1)[K_1]\;\cup\;\bigcup_{k=B-A+2}^{B-3}0.(B-1)(A-2)k[K_1\cup K_2]\\\\
&&\cup\;\bigcup_{p\geq1}0.(B-1)(A-2)((B-A+1)(A-2))^p(A-2)[K_1]\\\\
&&\hspace{1cm}\cup\;\bigcup_{p\geq1}\bigcup_{k=A-1}^{B-1}0.(B-1)(A-2)((B-A+1)(A-2))^pk[K_1\cup K_2]\\\\
&&\cup\;\bigcup_{p\geq1}\bigcup_{k=0}^{A-3}0.(B-1)(A-2)(B-A+1)((A-2)(B-A+1))^pk[K_4\cup K_5]\\\\
&&\hspace{1cm}\cup\;\bigcup_{p\geq1}0.(B-1)(A-2)(B-A+1)((A-2)(B-A+1))^p(A-2)[K_4]\\\\
&&\cup\;\{0.(B-1)\overline{(A-2)(B-A+1)}\}.
\end{array}
$$
Finally, 
$$\begin{array}{rcl}
\alpha_{B} &=&\{0.(B-1)(B-1)0\overline{0(B-1)}\}\\\\
&&\cup\; \bigcup_{p\geq0}\bigcup_{k=1}^{A-2}0.(B-1)(B-1)0(0(B-1))^p k[K_4\cup K_5]\\\\
&&\hspace{1cm}\cup\;\bigcup_{p\geq0}0.(B-1)(B-1)0(0(B-1))^p (A-1)[K_4]
\\\\
&&\cup\; \bigcup_{p\geq0}0.(B-1)(B-1)00((B-1)0)^p (B-A)[K_1]  \\\\
&&\hspace{1cm}\cup\;   \bigcup_{p\geq0}\bigcup_{k=B-A+1}^{B-2}0.(B-1)(B-1)00((B-1)0)^p k[K_1\cup K_2]
\\\\
&&\cup\;\bigcup_{p\geq0}0.(B-1)((B-A+1)(A-2))^p(B-A+1)[K_1]\\\\
&&\hspace{1cm}\cup\;\bigcup_{p\geq0}\bigcup_{k=B-A+2}^{B-1}0.(B-1)((B-A+1)(A-2))^pk[K_1\cup K_2]\\\\
&&\cup\;\bigcup_{p\geq0}\bigcup_{k=0}^{A-3}0.(B-1)(B-A+1)((A-2)(B-A+1))^pk[K_4\cup K_5]\\\\
&&\hspace{1cm}\cup\;\bigcup_{p\geq0}0.(B-1)(B-A+1)((A-2)(B-A+1))^p(A-2)[K_4]\\\\
&&\cup\;\{0.(B-1)\overline{(B-A+1)(A-2)}\}.
\end{array}
$$
\end{lemma}
\begin{proof}
Since $\alpha_1=C([s_1,t_1])$, this boundary part is described in $G^o$ as the set of infinite walks $w$ satisfying
$$(6;\mathbf{2(B-A)-2},\mathbf{1},\mathbf{B-2},\overline{\mathbf{2}})
\;\leq_{\textrm{lex}} w\;\leq_{\textrm{lex}} (6;\mathbf{2(B-A)-1},\mathbf{2A-2},\mathbf{4},\overline{\mathbf{2}})
$$ 
(see Table~\ref{tab:alpha} for the walks corresponding to the parameters $s_1$ and $t_1$). This splits into different cases:
$$w=(6;\mathbf{2(B-A)-2},\mathbf{1},\mathbf{B-2},\overline{\mathbf{2}})
$$
or 
$$w=(6;\mathbf{2(B-A)-2},\mathbf{1},\mathbf{B-2},\mathbf{2}^p,\mathbf{o},\ldots)\;\;(p\geq 0,\mathbf{o}\in\{\mathbf{3,4,\ldots,2A-1}\})
$$
or 
$$w=(6;\mathbf{2(B-A)-1},\mathbf{o},\ldots)\;\;(\mathbf{o}\in\{\mathbf{1,2,\ldots,2A-3}\})
$$
or
$$w=(6;\mathbf{2(B-A)-1},\mathbf{2A-2},\mathbf{o},\ldots)\;\;(\mathbf{o}\in\{\mathbf{1,2,3}\})
$$
or
$$w=(6;\mathbf{2(B-A)-1},\mathbf{2A-2},\mathbf{4},\mathbf{2^p},\mathbf{1},\ldots)\;\;(p\geq 0)
$$
or
$$(6;\mathbf{2(B-A)-1},\mathbf{2A-2},\mathbf{4},\overline{\mathbf{2}}).$$

Reading the corresponding sequences of digits leads to the description of $\alpha_1$. All the other cases are treated similarly.
\end{proof}
\begin{remark}\label{rem:descalphaprime}
A similar description is obtained for the curves $\alpha_i'$ ($i\in\{1,\ldots,B\}$) by changing every digit $a$ to the digit $B-1-a$ and by interchanging the sets $K_1\leftrightarrow K_4$,  $K_2\leftrightarrow K_5$ and $K_3\leftrightarrow K_6$.
\end{remark}
\begin{lemma}\label{lem:alphachain}
The curves $\alpha_1,\alpha_2,\ldots,\alpha_B$ form a chain.
\end{lemma}
\begin{proof}
The proof is based on Lemmata~\ref{lem:subd} and~\ref{lem:descalpha}. By Lemma~\ref{lem:descalpha}, $\alpha_i\subset \T_i$ for all $1\leq i\leq B-1$ and $\alpha_B\subset\T_{B-1}$. On the other hand, by Lemma~\ref{lem:subd},  $\T_i\cap\T_j=\emptyset$ as soon as $|i-j|\geq 2$. Therefore,  we  deduce that 
$\alpha_i\cap \alpha_j=\emptyset$ for $|i-j|\geq 2$ with $1\leq i,j\leq B-1$, and also $\alpha_B\cap\alpha_i=\emptyset$ for $1\leq i\leq B-3$. The  proof for the remaining case $\alpha_B\cap\alpha_{B-2}=\emptyset$ can be found in Appendix~\ref{app:alphachain}.

Let us show that $\alpha_2\cap\alpha_1$ reduces to a single point. By Lemma~\ref{lem:descalpha}, we have the following inclusions.
$$\alpha_1\subset\bigcup_{k=B-A+1}^{B-1}\T_{10k}\cup\T_{10(B-A)0}\cup\bigcup_{k=1}^{A-2}\T_{1k},
$$
and
$$\alpha_2\subset\bigcup_{k=0}^{A-3}\T_{2k}\cup\T_{2(A-2)(B-1)}\cup\T_{2(A-2)(B-2)0}.
$$
However, by Lemma~\ref{lem:subd}, to have $\T_{2a_2a_3a_4}\cap\T_{1a_2'a_3'a_4'}\ne\emptyset$ requires that 
$$\left\{\begin{array}{c}a_2-a_2'=A\\a_3-a_3'=B-1\end{array}\right.
\textrm{ or }\left\{\begin{array}{c}a_2-a_2'=A-1\\a_3-a_3'\in\{B-A,B-A+1,B-A+2\}\end{array}\right.\textrm{ or }\left\{\begin{array}{c}a_2-a_2'=A-2\\a_3-a_3'=-1\end{array}\right..
$$
Therefore,
$$\alpha_2\cap\alpha_1\subset\T_{2(A-2)(B-2)0}\cap\T_{10(B-1)}=\{0.2(A-2)(B-2)\overline{0(B-1)}\}=\{0.10(B-1)\overline{(B-1)0}\}.
$$
The above equality follows again from Lemma~\ref{lem:subd}. This point set is itself included in $\alpha_1\cap\alpha_2$ by Lemma~\ref{lem:descalpha}.

Now, it can be seen from Lemma~\ref{lem:descalpha} that, for $i\in\{2,\ldots,B-2\}$, $\alpha_i$ is just a translate of $\alpha_1$, by the vector $0.(i-1)$. In this way, we obtain easily that 
$$\alpha_i\cap \alpha_{i+1}=\{0.(i+1)(A-2)(B-2)\overline{0(B-1)}\}=\{0.i0(B-1)\overline{(B-1)0}\}
$$
for all $i\in\{1,\ldots,B-3\}$. 

We now show that  $\alpha_{B-2}\cap\alpha_{B-1}$ reduces to a point. By Lemma~\ref{lem:descalpha}, we have the following inclusions.
$$\alpha_{B-2}\subset \T_{(B-2)0(B-1)}\;\cup\;0.(B-2)0[K_4\cup K_5]\;\cup\;\bigcup_{k=1}^{A-2}\T_{(B-2)k}
$$
and 
$$
\alpha_{B-1}\subset\T_{(B-1)(A-2)}.
$$
By Lemma~\ref{lem:subd}, $\T_{(B-2)0(B-1)}\cap\T_{(B-1)(A-2)}$ reduces to one point, while $\T_{(B-2)k}\cap\T_{(B-1)(A-2)}=\emptyset$ for all $k\geq1$. Therefore, we obtain that 
$$
\begin{array}{rcl}
\alpha_{B-2}\cap\alpha_{B-1}&\subset&\{0.(B-2)0(B-1)\overline{(B-1)0}\}\;\cup\;0.(B-2)0[K_4\cup K_5]\cap \alpha_{B-1}.
\end{array}
$$
Again by Lemma~\ref{lem:subd}, since $\alpha_{B-1}\subset\T_{(B-1)(A-2)}$, the set $0.(B-2)0[K_4\cup K_5]\cap \alpha_{B-1}$ is a subset of the finite set $\bigcup_{a_3'=1}^{B-1}\{0.(B-2)0a_3'\overline{(B-1)0}\}$. Let $a_3'\in\{1,\ldots, B-1\}$. Let us show that the point $0.(B-2)0a_3'\overline{(B-1)0}$ does not belong to $0.(B-2)0[K_4\cup K_5]$. Indeed, 
$$0.(B-2)0[K_4\cup K_5]=0.(B-2)0(B-1)[K_6]\cup\bigcup_{k=B-A+1}^{B-1}0.(B-2)0k[K_1\cup K_2]\cup0.(B-2)0(B-A)[K_1].
$$
However, by Lemma~\ref{lem:subd}, it is easy to check that $\overline{(B-1)0}$ does not belong to $K_6\cup K_1\cup K_2$.

It follows that
$$ 
\alpha_{B-2}\cap\alpha_{B-1}\subset\{0.(B-2)0(B-1)\overline{(B-1)0}\}=\{0.(B-1)(A-2)(B-2)\overline{0(B-1)}\}
$$
and this point set is itself included in $\alpha_{B-2}\cap\alpha_{B-1}$ by Lemma~\ref{lem:descalpha}.

The case $\alpha_{B-1}\cap\alpha_{B}$ can be found in Appendix~\ref{app:alphachain}.
\end{proof}

\subsection{The curves $\alpha_1,\ldots,\alpha_B,\alpha_1',\ldots,\alpha_B'$ form a circular chain}\label{sub:circchain}

\begin{lemma}\label{lem:alphaprimchain}
The curves $\alpha_1',\alpha_2',\ldots,\alpha_B'$ form a chain.
\end{lemma}
\begin{proof}
This property is deduced from Lemma~\ref{lem:alphachain}, stating that the curves $\alpha_1,\alpha_2,\ldots,\alpha_B$ form a chain. Indeed, given two infinite sequence of digits $(a_n)_{n\in\mathbb{N}}$ and $(a_n')_{n\in\mathbb{N}}$, we have
$$\begin{array}{rcl}
0.a_1a_2\ldots&=&0.a_1'a_2'\ldots\\
\iff 0.\overline{B-1}-0.a_1a_2\ldots&=&0.\overline{B-1}-0.a_1'a_2'\ldots\\
\iff 0.(B-1-a_1)(B-1-a_2)\ldots&=&0.(B-1-a_1')(B-1-a_2')\ldots
\end{array}
$$
Hence, by the definition of the curves $\alpha_i'$ (see Section~\ref{sub:constrQ}), we have for all $i,j\in\{1,\ldots, B\}$ with $|i-j|\geq 2$ that
$$\alpha_{i}'\cap\alpha_{j}'=\emptyset
$$
and for all $i\in\{1,\ldots, B-1\}$ that 
$$\alpha_{i+1}'\cap\alpha_i'=\{0.(B-1-(i+1))(B-1-(A-2))(B-1-(B-2))\overline{(B-1)0}\}=\{0.(B-1-i)(B-1)0\overline{0(B-1)}\}.
$$
\end{proof}
\begin{lemma}\label{lem:alphaalphaprimeemptyintersec} For all $(j,k)\notin\{(B,1),(1,B)\}$,
$$\alpha_j\cap\alpha_k'=\emptyset.
$$
\end{lemma}
\begin{proof}\emph{In the first part of the proof}, we suppose that $j,k\in\{1,\ldots,B-2\}$.

Since $\alpha_j\subset\T_j$, while $\alpha_k'\subset\T_{B-1-k}$, we deduce that 
$$\alpha_j\cap\alpha_k'=\emptyset 
$$
for all $j,k$ satisfying $|j-(B-1-k)|\notin\{0,1\}$.

Let us consider the case $j-(B-1-k)=0$, that is, $k=B-1-j$.  We prove the result for $j=1$, \emph{i.e.}, we show that  
$$\alpha_1\cap\alpha_{B-2}'=\emptyset.
$$
The same will hold for $\alpha_j\cap\alpha_{B-1-j}'$ for all $j\in\{2,\ldots,B-2\}$, since the sets $\alpha_j$ (resp. $\alpha_{B-1-j}$) are described by sets of points whose expansions  only differ from their first digit, $j$.

By Lemma~\ref{lem:descalpha}, 
$$\begin{array}{rcl}
\alpha_1&\subset&\left(\bigcup_{a_1=0}^{A-3}\bigcup_{a_2=B-A}^{B-1}\T_{1a_1a_2}\right)\;
\cup\;\T_{1(A-2)(B-1)}\;\cup\;\T_{1(A-2)(B-2)0}
\end{array}
$$
and 
$$\begin{array}{rcl}
\alpha_{B-2}'&\subset& \left(\bigcup_{a_1'=B-A+2}^{B-1}\bigcup_{a_2'=0}^{A-1}\T_{1a_1'a_2'}\right)\;\cup\;\T_{1(B-A+1)0}\;\cup\;\T_{1(B-A+1)1(B-1)}.
\end{array}
$$
We check that the sets on the right sides sets are pairwise disjoint, using Lemma~\ref{lem:subd}.
\begin{itemize}
\item[$(i)$] 
$$\left(\bigcup_{a_1=0}^{A-3}\bigcup_{a_2=B-A}^{B-1}\T_{1a_1a_2}\right)\;\cap\; \left(\bigcup_{a_1'=B-A+2}^{B-1}\bigcup_{a_2'=0}^{A-1}\T_{1a_1'a_2'}\right)=\emptyset,
$$
because $a_1'-a_1\geq B-A+2-(A-3)=2$ for all $a_1,a_1'$ showing up in these unions.

\item[$(ii)$] $\T_{1(A-2)(B-1)}\;\cap\;\left(\bigcup_{a_1=0}^{A-3}\bigcup_{a_2=B-A}^{B-1}\T_{1a_1a_2}\right)=\emptyset$. Indeed, since $a_1'\geq B-A+2$, we have that $|a_1'-(A-2)|\leq 1$ iff $a_1'=B-A+2$, and in this case the difference is $1$. Now, 
$$\T_{1(A-2)(B-1)}\cap \bigcup_{a_2'=0}^{A-1}\T_{1(B-A+2)a_2'}=\emptyset,
$$
because $a_1'-a_1=1$ but $a_2'-a_2\leq0$. 

\item[$(iii)$] 
$\T_{1(A-2)(B-2)0}\;\cap\;\left(\bigcup_{a_1=0}^{A-3}\bigcup_{a_2=B-A}^{B-1}\T_{1a_1a_2}\right)=\emptyset.
$
Similarly to Item~$(ii)$, we can restrict to the intersection
$$\T_{1(A-2)(B-1)0}\cap \bigcup_{a_2'=0}^{A-1}\T_{1(B-A+2)a_2'},
$$
which is empty, since $a_1'-a_1=1$, but $a_2'-a_2\leq A-1-(B-1)=4-A<A-2$ (we exclude $A=3=B$).

\item[$(iv)$]
$$\left(\bigcup_{a_1=0}^{A-3}\bigcup_{a_2=B-A}^{B-1}\T_{1a_1a_2}\right)\;\cap\; \T_{1(B-A+1)0}=\emptyset.
$$
Indeed, since $a_1\leq A-3$, we have that $|a_1-(B-A+1)|\leq 1$ iff $a_1=A-3$, and in this case the difference is $1$. Now, 
$$\bigcup_{a_2=B-A}^{B-1}\T_{1(A-3)a_2}\;\cap\;\T_{1(B-A+1)0} =\emptyset,
$$
because $a_1-a_1'=1$ but $a_2-a_2'\leq0$. 

\item[$(v)$]
$$\left(\bigcup_{a_1=0}^{A-3}\bigcup_{a_2=B-A}^{B-1}\T_{1a_1a_2}\right)\;\cap\; \T_{1(B-A+1)1(B-1)}=\emptyset.
$$
Similarly to Item~$(iv)$, we can restrict to the intersection
$$\bigcup_{a_2=B-A}^{B-1}\T_{1(A-3)a_2}\;\;\T_{1(B-A+1)1(B-1)},
$$
which is empty, since $a_1-a_1'=1$, but $a_2-a_2'\leq 1-(B-A)=4-A<A-2$.

\item[$(vi)$] $\T_{1(A-2)(B-1)}\cap\T_{1(B-A+1)0}=\emptyset$
since $A-2=(B-A+1)$ and $B-1-0\geq2$.

\item[$(vii)$] $\T_{1(A-2)(B-1)}\cap\T_{1(B-A+1)1(B-1)}=\emptyset$
since $A-2=(B-A+1)$ and $B-1-1=B-2\geq2$.

\item[$(viii)$] $\T_{1(A-2)(B-2)0}\cap\T_{1(B-A+1)0}=\emptyset$
since $A-2=(B-A+1)$ and $B-2-0=B-2\geq2$.

\item[$(ix)$] $\T_{1(A-2)(B-2)0}\cap\T_{1(B-A+1)1(B-1)}=\emptyset$
since $A-2=(B-A+1)$ and $B-2-1=B-3\geq2$.
\end{itemize}
Therefore, $\alpha_1\cap\alpha_{B-2}'=\emptyset$, and consequently $\alpha_j\cap\alpha_{B-1-j}'$ for all $j\in\{1,\ldots,B-2\}$, as mentioned earlier.  

Let us now consider the case $j-(B-1-k)=1$, that is, $k=B-j$. Again, it is sufficient to prove the result for $j=2$, then it will follow that $\alpha_j\cap\alpha_{B-j}'=\emptyset$ for all $j\in\{2,\ldots, B-2\}$. Since 
$$\alpha_{B-2}'\subset\bigcup_{a_2'=B-A+1}^{B-1}\T_{1a_2'}
$$
and 
$$\alpha_2\subset\bigcup_{a_2=0}^{A-2}\T_{2a_2},
$$
the intersection $\alpha_{B-2}'\cap\alpha_2$ is contained in sets $\T_{1a_2'}\cap\T_{2a_2}$ satisfying $a_2-a_2'\in\{A,A-1,A-2\}$ (see Lemma~\ref{lem:subd}). In particular, $a_2-a_2'>0$. However, we can read off from Lemma~\ref{lem:descalpha} that $a_2\leq A-2$ and $a_2'\geq B-A+1$, hence $a_2-a_2'\leq 2A-B-3=0$. Hence $\alpha_2\cap\alpha_{B-2}'=\emptyset$.

Let us finally consider the case $j-(B-1-k)=-1$, that is $k=B-2-j$. It is sufficient to prove the result for $j=1$, then  it follows that $\alpha_j\cap\alpha_{B-2-j}'=\emptyset$ holds for all $j\in\{1,\ldots, B-3\}$. The case $j=1$ can be found in Appendix~\ref{app:alphaalphaprimeemptyintersec}.

\emph{In the second part of the proof}, we deal with the remaining cases. 
Let us consider the case $j=B-1$ and $k\in\{1,\ldots, B\}$. Since $\alpha_{B-1}\subset\T_{B-1}$, $\alpha_{k}'\subset\T_{B-1-k}$ for $k\in\{1,\ldots,B-1\}$ and $\alpha_{B}'\subset\T_0$, we have 
$$\alpha_{B-1}\cap\alpha_k'=\emptyset
$$
for all $k\in\{2,\ldots,B\}$. We now prove that $\alpha_{B-1}\cap\alpha_1'=\emptyset$. We notice that
$$\alpha_{B-1}\subset\T_{(B-1)(A-2)}
$$
and 
$$\alpha_{1}'\subset\bigcup_{k=A-2}^{B-1}\T_{(B-2)k}.
$$
However, by Lemma~\ref{lem:subd}, $\T_{(B-1)(A-2)}\cap\T_{(B-2)k}$ for all $k\in\{A-2,\ldots,B-1\}$. Therefore, $\alpha_{B-1}\cap\alpha_1'=\emptyset$.

It follows immediately, by symmetry, that  
$$\alpha_{B-1}'\cap\alpha_k=\emptyset
$$
for all $k\in\{2,\ldots,B\}$.

Finally, since $\alpha_B\subset\T_{B-1}$, $\alpha_k'\subset\T_{B-1-k}$ for $k\in\{1,\ldots, B-1\}$ and $\alpha_{B}'\subset\T_0$, we have
$$\alpha_{B}\cap\alpha_k'=\emptyset
$$
for all $k\in\{2,\ldots,B\}$. And symmetrically, it follows that
$$\alpha_{B}'\cap\alpha_k=\emptyset
$$
for all $k\in\{2,\ldots,B\}$.
\end{proof}

\begin{lemma}\label{lem:circchain}
The curves $\alpha_1,\ldots,\alpha_B,\alpha_1',\ldots,\alpha_B'$ form a circular chain.
\end{lemma}
\begin{proof} We prove that 
$$\alpha_B\cap\alpha_1'=\{0.(B-1)(B-1)0\overline{0(B-1)}\}=\{0.(B-2)(A-2)1\overline{(B-1)0}\},
$$
$$\alpha_B'\cap\alpha_1=\{0.00(B-1)\overline{(B-1)0}\}=\{0.1(A-2)(B-2)\overline{0(B-1)}\}
$$
and the result then follows from Lemmata~\ref{lem:alphachain},~\ref{lem:alphaprimchain} and~\ref{lem:alphaalphaprimeemptyintersec}.

 Note that a convenient  expression for $\alpha_1'$ and $\alpha_B'$ can be obtained as stated in Remark~\ref{rem:descalphaprime}. 

We start with the first intersection. Let  $0.(B-2)a_2'a_3'\cdots\in\alpha_1'$ and $0.(B-1)a_2a_3\cdots\in\alpha_B$ with expansions of the form given in Lemma~\ref{lem:descalpha} and Remark~\ref{rem:descalphaprime}. Then, $a_2\leq B-1$ and $a_2'\geq B-A+1=A-2$, thus $a_2-a_2'\leq A-2$. By Lemma~\ref{lem:subd}, for these expansions to be equal we must have $a_2-a_2'\in\{A,A-1-A-2\}$, therefore we can deduce that $a_2-a_2'=A-2$, which happens only for $a_2=B-1$ and $a_2'=A-2$. Furthermore,
again by Lemma~\ref{lem:subd}, we must have $a_4a_5a_6a_7\cdots=\overline{0(B-1)}$. One easily checks in the expression of $\alpha_B$ in Lemma~\ref{lem:descalpha} that the only expansion of this kind is that of 
$$0.(B-1)(B-1)0\overline{0(B-1)}=0.(B-2)(B-A+1)1\overline{(B-1)0}.
$$
It follows that 
$$
\alpha_B\cap\alpha_1'=\{0.(B-1)(B-1)0\overline{0(B-1)}\}.
$$

For the second intersection, we just use the definition of $\alpha_j'$ for $j\in\{1,B\}$: it is obtained from $\alpha_j$ by changing every digit $a$ to $B-1-a$. Hence the result follows from the computation  of $\alpha_B\cap\alpha_1'$.
\end{proof}

\subsection{For each $i\in\{1,\ldots, B-2\}$, $\gamma_i$ is a simple arc on $\partial \T$ going through the point $0.i\overline{A-2}$ and with endpoints on $Q'$}\label{sub:gamma}
 Note that $\beta_{B-1}$ and $\beta_{B}$ are two consecutive simple arcs of a chain, with intersection point $0.(B-1)\overline{A-2}$ (see Table~\ref{tab:alpha}). The endpoints of the simple arc $\beta_{B-1}\cup\beta_{B}$ are 
$$0.(B-1)(A-2)(B-2)\overline{0(B-1)},\;0.(B-1)(B-1)0\overline{0(B-1)}.
$$
Therefore, by construction, for each $i\in\{1,\ldots, B-2\}$, the curve $\gamma_i$ is a simple arc, going through the point $0.i\overline{A-2}$ and having endpoints
$$0.i(A-2)(B-2)\overline{0(B-1)},\;0.i(B-1)0\overline{0(B-1)}.
$$
As can be seen on Table~\ref{tab:alpha}, the points $0.i(A-2)(B-2)\overline{0(B-1)}$ are endpoints of $\alpha_i$, hence of $\beta_i$ also, and thus belong to $Q'$. Also, as the points $0.i0(B-1)\overline{(B-1)0}$ are endpoints of $\alpha_i$, the  following points,
$$0.(B-1-i)(B-1)0\overline{0(B-1)},
$$
 obtained from the former ones by exchanging the digits $a\leftrightarrow B-1-a$, are endpoints of $\alpha_{B-1-i}'$, hence of $\beta_{B-1-i}'$ and thus belong to $Q'$. Consequently, the endpoints of $\gamma_i$ all lie on $Q'$.

Finally, we claim that $\gamma_i\subset\partial\T$ for each $i\in\{1,\ldots, B-2\}$. Let us observe that 
$$\beta_{B-1}\subset\alpha_{B-1}\subset0.(B-1)[K_2]\subset\partial \T.
$$
The inclusion for $\alpha_{B-1}$  can be checked from Table~\ref{tab:alpha} and the automaton on the right side of Figure~\ref{CNS_DTautomaton}. We deduce from the same automaton that 
$0.i[K_2]\subset K_6\subset\partial\T$ for $i\in\{0,\ldots,B-A\}$ and $0.i[K_2]\subset K_5\subset\partial\T$
 for $i\in\{B-A+1,\ldots,B-1\}$. 
It follows that 
$$f_if_{B-1}^{-1}(\beta_{B-1})\subset0.i[K_2]\subset\partial\T
$$ for all  $i\in\{1,\ldots, B-2\}$.

In the same way, we can observe that 
$$\beta_{B}\subset\alpha_{B}\subset0.(B-1)[K_4]\cup0.(B-1)[K_5]\subset\partial\T.
$$
Moreover, $0.i[K_4]\subset K_2\subset\partial\T$ for $i\in\{0,\ldots,A-1\}$, $0.i[K_4]\subset K_3\subset\partial\T$
 for $i\in\{A,\ldots,B-1\}$, and further that $0.i[K_5]\subset K_2\subset\partial\T$ for $i\in\{0,\ldots,A-2\}$, $0.i[K_5]\subset K_3\subset\partial\T$  for $i\in\{A-1,\ldots,B-1\}$.
It follows that 
$$f_if_{B-1}^{-1}(\beta_{B})\subset0.i[K_4]\cup0.i[K_5]\subset\partial\T
$$ for all  $i\in\{1,\ldots, B-2\}$.
Therefore, $\gamma_i\subset\partial\T$ for all $i\in\{1,\ldots, B-2\}$.

\subsection{$\T$ has no cut point for $2A-B=3$, $A\ne B$}\label{sub:checkassu}

\begin{theorem}\label{th:ii3} Let $M\in\mathbb{Z}^{2\times 2}$ with  characteristic polynomial $x^2+Ax+B$, where $0\leq A\leq B\geq 2$.  Let $\Di=\{0,v,2v,\ldots,(B-1)v\}$ for some $v\in\mathbb{Z}^2$ such that $v,M v$ are linearly independent. Denote by $\T=\T(M,\Di)$ the associated self-affine tile. Suppose that $2A-B=3$ and $A\ne B$. Then $\T$ has no cut point but its interior is disconnected. The closure of each connected component of the interior of $\T$ is homeomorphic to the closed disk.
\end{theorem}
\begin{proof}  We check that $Q\subset\partial \T$ satisfies the assumptions of Theorem~\ref{th:nt04}. By construction, $Q$ is connected. It has  no cut point because $Q'$ is a simple closed curve and the $\gamma_i$'s ($i\in\{1,\ldots, B-2\}$) are simple arcs with endpoints on $Q'$. Moreover, for each $i\in\{0,\ldots,B-1\}$, $f_i(Q)\cap Q$ contains the point
$$P_i=0.i\overline{A-2}=f_i(0.\overline{A-2}).
$$
This holds because $S=0.\overline{A-2}\in Q$ and, for $i\in\{1,\ldots, B-2\}$, $0.i\overline{A-2}\in\gamma_i\subset Q$. For $i=B-1$, the point $0.(B-1)\overline{A-2}$ is an endpoint of $\beta_{B-1}$ by construction, thus it belongs to $Q$. For $i=0$, the point $0.0\overline{A-2}=0.\overline{B-1}-0.(B-1)\overline{A-2}$ is an endpoint of $\beta_{B-1}'$, thus it also belongs to $Q$.

We finally find another point in this intersection\footnote{Most of the following argument can be found in~\cite{NgaiTang04} for the special case $A=4,B=5$}. By construction, the simple closed curve $Q'$ is the union of two simple arcs meeting at their endpoints, namely $\bigcup_{i=1}^B\beta_i$ and $\bigcup_{i=1}^B\beta_i'=0.\overline{B-1}-\bigcup_{i=1}^B\beta_i$. It follows that $Q'$ admits 
$$\frac{1}{2}0.\overline{B-1}=0.\overline{A-2}=S
$$
as a symmetry center. 
Hence we conclude that $S$ lies in the bounded component $B_c$ of $\mathbb{R}^2\setminus Q'$. Therefore, $f_i(S)=0.i\overline{A-2}\in Q$ lies in the bounded component $f_i(B_c)$ of $\mathbb{R}^2\setminus f_i(Q')$. Since $f_i$ is a contraction and $Q'=\partial B_c$,  there exists a point $z\in Q'\setminus f_i(B_c)$. Now, as $Q$ is arcwise connected, there is a path in $Q$ from $z$ to $f_i(S)$. This path must intersect $\partial f_i(B_c)=f_i(Q')$, providing a point in $Q\cap f_i(Q)$, different from $f_i(S)$.


It follows from Theorem~\ref{th:nt04} that $\T$ has no cut point for $2A-B=3$, $A\ne B$ and that the closure of each interior component of $\T$
is a topological disk.
\end{proof}

\begin{remark} Note that the disk-likeness of the closure of the interior components follows from the following theorem of Torhorst.
\begin{lemma}\cite[$\S$61, II, Theorem 4]{Kuratowski68}\label{thorhorst}
Let $M\subset \mathbb{S}^2 $ be a locally connected continuum
having no cutpoint and $R$ a component of $M^c$. Then
$\overline{R}$ is homeomorphic to a disk.
\end{lemma}
\end{remark}
\end{section}

\begin{section}{$\T$ has no cut point for $2A-B=4$}\label{sec:thnocp4} 

The detailed proof of this case will be published in a separate paper~\cite{Loridant0000}. We list in this section the two  main differences between the cases $2A-B=3$ and $2A-B=4$. 

\begin{itemize}
\item There are significant changes in Lemma~\ref{lem:subd}, describing the intersecting natural subdivisions of the IFS-attractor $\T$. In particular, in the case $2A-B=3$, it follows from Lemma~\ref{lem:subd} that every intersection $\T_{a_1a_2}\cap\T_{a_1'a_2'}$ contains exactly $B-1$ points as soon as $a_1-a_1'=1$ and $a_2-a_2'=A-2$. These points build bridges between subdivisions of $\T$. These bridges are ``far enough from each other'' for the curve $Q'\subset Q$ of Section~\ref{sub:constrQ} to be self-avoiding.  Whenever $2A-B=4$, a lemma  analogous to Lemma~\ref{lem:subd} shows that every intersection $\T_{a_1a_2}\cap\T_{a_1'a_2'}$ contains infinitely many points as soon as $a_1-a_1'=1$ and $a_2-a_2'=A-2$. This involves changes in the construction of the curve $Q$   in order to make the subcurve $Q'$ be a simple closed curve.   
\item In the case $2A-B=3$, we have $(B-1)/2=A-2$ and the symmetry point $\frac{1}{2}0.\overline{B-1}$ of $\T$ has the alternative expansion $0.\overline{A-2}$. This point therefore belongs to the curve $Q$ and is in the bounded component of $\mathbb{R}^2\setminus Q'$ (see the proof of Theorem~\ref{th:ii3}). This alternative expansion is not available whenever $2A-B=4$. Consequently, for $Q$ to satisfy  $\# f_i(Q)\cap Q\geq 2$  for all $i\in\{1,\ldots,B\}$ (see assumption of Theorem~\ref{th:nt04}), our idea  will be to add more simple arcs to the subcurve $Q'$ (see Section~\ref{sub:constrQ}).  
\end{itemize}

\end{section}

\newpage
\appendix
\section{Complements to the proof of Lemma~\ref{lem:alphachain}}\label{app:alphachain} Here $2A-B=3$ and $A\ne B$. We use again frequently Lemmata~\ref{lem:subd} and~\ref{lem:descalpha}.
\subsection{Proof that $\alpha_B\cap\alpha_{B-2}=\emptyset$.}
We have $\alpha_{B-2}\subset\T_{B-2}$ and $\alpha_{B}\subset\T_{B-1}$. Suppose that $0.(B-2)a_2'a_3'\cdots[K']\cap0.(B-1)a_2a_3[K]\ne\emptyset$ for some subsets of $\alpha_{B-2}$ and $\alpha_{B}$ listed in Lemma~\ref{lem:descalpha}. By Lemma~\ref{lem:subd}, we necessarily have $a_2-a_2'\in\{A,A-1,A-2\}$. However, if $a_2-a_2'=A$, then $a_3'=0$, a contradiction to Lemma~\ref{lem:descalpha}. Hence $a_2-a_2'\in\{A-1,A-2\}$. In view of Lemma~\ref{lem:descalpha}, only the following cases meet this condition. 
\begin{itemize}
\item[$(i)$]$a_2=B-1,a_2'=A-2$, hence $a_2-a_2'=A-2$. Then, either $a_3=0, a_3'\in\{B-1,B-2\}$, and $a_3-a_3'\ne-1$, contradicting Lemma~\ref{lem:subd}. Or $a_3\in\{1,\ldots,A-1\}$ and $a_4\geq A-3$, contradicting Lemma~\ref{lem:subd} again.
\item[$(ii)$] $a_2=B-1,a_2'=A-3$, that is, we consider the subsets $0.(B-2)(A-3)[K_4\cup K_5]\subset\alpha_{B-2}$. Here, $a_2-a_2'=A-1$. If $a_3=0$, then $a_3-a_3'\leq 0$, contradicting Lemma~\ref{lem:subd}. Hence we are left with the subsets $0.(B-1)(B-1)[K_1\cup K_2]\subset\alpha_B$.  For the subsets under consideration, we have $a_3-a_3'\leq A-1-(A-3)=2$. By Lemma~\ref{lem:subd}, we are interested in the cases where $a_3-a_3'\in\{B-A,B-A+1,B-A+2\}$. This leads to three cases. Either $a_3=A-1,a_3'=A-2$, hence $a_3-a_3'=1$. Thus $B-A=1$, that is, $A=4,B=5$ and $a_4=0$, contradicting $a_4\geq A-3$ (Lemma~\ref{lem:subd}); or $B-A+1=1$, that is, $A=B$, which is not considered here. Or $a_3=A-1,a_3'=A-3$, hence $a_3-a_3'=2=B-A+2$ and $A=B$, which is not considered here. Or $a_3=A-2,a_3'=A-3$, hence $a_3-a_3'=1$, which brings back to the first case.
\item[$(iii)$]$a_2=A-2,a_2'=0$. Then $a_2-a_2'=A-2$ but from Lemma~\ref{lem:descalpha} we read that $a_4\ne0$, contradicting Lemma~\ref{lem:subd}.
\item[$(iv)$]$a_2\notin\{B-1,A-2\}$, that is, we consider the subsets $0.(B-1)a_2[K_1\cup K_2]\subset\alpha_B$ with $a_2\in\{A-1,\ldots,B-2\}$. This implies that $a_2'\in\{0,\ldots,A-3\}$ in order to have $a_2-a_2'\in\{A-1,A-2\}$. If $a_2'=0$, then only $a_2=A-1$ has to be considered. In this case, if $a_3'=B-1$, then $a_3-a_3'\leq 0$, which contradicts Lemma~\ref{lem:subd}. Therefore, we are left with the subsets $0.(B-2)a_2'[K_4\cup K_5]$ with $a_2'\in\{0,\ldots,A-3\}$. If now $a_2-a_2'=A-2$, then $a_4=0$ by Lemma~\ref{lem:subd}, a contradiction to the fact that $a_4\geq A-3$ for the subsets of $\alpha_B$ under consideration. On the other side, if $a_2-a_2'=A-1$, the situation is analogous to that of Item~$(ii)$ and this leads to a contradiction.
\end{itemize}
Therefore, $\alpha_B\cap\alpha_{B-2}=\emptyset$.

\subsection{Proof that $\alpha_B\cap\alpha_{B-1}=\{0.(B-1)\overline{A-2}\}$.} The curves $\alpha_{B-1}$ and $\alpha_B$ are both subsets of $\T_{B-1}$. From $2A-B=3$ it follows that $A-2=B-A+1$ and by Lemma~\ref{lem:subd} the point $0.(B-1)\overline{A-2}$ belongs to $\alpha_{B-1}\cap\alpha_B$. 

We will now show that $0.(B-1)(A-2)a_2'a_3'\cdots[K']\cap0.(B-1)a_1a_2a_3\cdots[K]=\emptyset$ for all the subsets of $\alpha_{B-1}$ and $\alpha_{B}$ listed in Lemma~\ref{lem:descalpha}. Let us suppose that this is not the case. This is equivalent to assuming that $0.(A-2)a_2'a_3'\cdots[K']\cap0.a_1a_2a_3\cdots[K]\ne\emptyset$. We review all such pairs of sets occurring in $\alpha_{B-1}$ and $\alpha_B$. Firstly, we consider the cases where $a_1\ne A-2$.
\begin{itemize}
\item[$(i)$] If  $a_1\in\{A,\ldots,B-1\}$, then $a_1-(A-2)\geq 2$ (because $A\ne B$ implies $A\geq 4$), and this contradicts Lemma~\ref{lem:subd}.  
\item[$(ii)$] If $a_1=A-1$, that is, we consider the subset $0.(A-1)[K_1\cup K_2]$, then $a_1-(A-2)=1$. However, we can see from Lemma~\ref{lem:descalpha} that $a_2\leq A-1$ and $a_2'\geq A-2$, thus $a_2-a_2'\leq 1$, contradicting Lemma~\ref{lem:subd}. 
\end{itemize}

Secondly, we come to the cases where $a_1=A-2$. Then our assumption is equivalent to $0.a_2'a_3'\cdots[K']\cap0.a_2a_3\cdots[K]\ne\emptyset$. Let us start with all the cases where $a_2'\ne A-2$.
\begin{itemize}
\item[$(i)$] If $a_2'\geq A-1$ and $a_2\leq A-3$, then $a_2-a_2'\leq -2$, contradicting the first item of Lemma~\ref{lem:subd}.
\item[$(ii)$] If $a_2'\geq A$ and $a_2=A-2$, then $a_2-a_2'\leq -2$, contradicting the first item of Lemma~\ref{lem:subd}.
\item[$(iii)$] If $a_2'=A-1$ and $a_2=A-2$, then $a_2'-a_2=1$. One checks on the expression of $\alpha_B$ that if $a_2=A-2$, then $a_3\geq A-2$. Moreover, the only subsets satisfying $a_2'=A-1$ are $0.(A-1)[K_1\cup K_2]$ and, in the case $A=4,B=5$, the subsets with $a_2'=B-2$. In both cases,  $a_3'\leq A-1$, but since $a_3\geq A-2\geq 2$, we obtain that  $a_3'-a_3\leq A-3$, contradicting Lemma~\ref{lem:subd}.  
\end{itemize}
We go on with the cases where $a_2'=A-2$. If $a_2=0$, then $a_2'-a_2=A-2\geq2$ and this contradicts Lemma~\ref{lem:subd}. Otherwise, $a_2=A-2=a_2'$ and our assumption is equivalent to $0.a_3'\cdots[K']\cap0.a_3\cdots[K]\ne\emptyset$. We list the different cases below. The sets are chosen according to Lemma~\ref{lem:descalpha}, starting from the digits $a_3$ and $a_3'$, since we erased $B-1$, $a_1=a_1'=A-2$ and $a_2=a_2'=A-2$.
\begin{itemize}
\item[$(i)$] Let $p\geq 0$ and $S:=0.((A-2)(A-2))^p(A-2)[K_1]\subset\T_{A-2}$.  We apply Lemma~\ref{lem:subd}. Since $S_1:=0.[K_1]\subset\T_{0}$, we have $S\cap S_1=\emptyset$. Also, $S$ does not intersect $S_2:=0.(A-2)((A-2)(A-2))^m(A-2)[K_1]$ for $m\geq0$. Indeed, if $p\geq m$, then $S\cap S_2\ne\emptyset$ iff $0.((A-2)(A-2))^{p-m}(A-2)[K_1]\cap 0.(A-2)(A-2)[K_1]\ne\emptyset$, implying that $\T_{(A-2)^{2(p-m)}0}\cap\T_{(A-2)0}\ne\emptyset$; but this does not occur by the first item of Lemma~\ref{lem:subd}.  And if $p<m$, then $S\cap S_2\ne\emptyset$ would imply that  $\T_{0}\cap\T_{(A-2)^{2(m-p)+1}}\ne\emptyset$, which again does not occur by Lemma~\ref{lem:subd}.  In the same way, we can prove that $S$ does not intersect the sets $S_3:=0.(A-2)((A-2)(A-2))^mk[K_1\cup K_2]$ ($0\leq k\leq A-3,m\geq 0$), $S_4:=0.((A-2)(A-2))^mk[K_4\cup K_5]$ ($0\leq k\leq A-3,m\geq 1$) and $S_5:=0.((A-2)(A-2))^m(A-2)[K_4]$ ($m\geq 1$).

\item[$(ii)$] Similarly as in Item~$(i)$, taking successively for $S$ the sets 
$$\begin{array}{l}
0.((A-2)(A-2))^pk[K_1\cup K_2] \;(A-1\leq k\leq B-1, p\geq0), \\
0.(A-2)((A-2)(A-2))^pk[K_4\cup K_5]\; (0\leq k\leq A-3, p\geq0), \textrm{ and}\\
0.(A-2)((A-2)(A-2))^p(A-2)[K_4] \;(p\geq0),
\end{array}
$$ 
one can check that $S\cap S_j=\emptyset$ for all the sets $S_j$ ($j\in\{1,\ldots,5\}$) defined in the above item.

\end{itemize}
One finally checks with the help of Lemma~\ref{lem:subd} that the point $0.(B-1)(A-2)(B-2)\overline{0(B-1)}$ does not lie on $\alpha_B$, and that the point $0.(B-1)(B-1)0\overline{0(B-1)}$ does not lie on $\alpha_{B-1}$. 

Therefore, $\alpha_B\cap\alpha_{B-1}=\{0.(B-1)\overline{A-2}\}$.

\section{Complements to the proof of Lemma~\ref{lem:alphaalphaprimeemptyintersec}}\label{app:alphaalphaprimeemptyintersec} Here $2A-B=3$ and $A\ne B$. \textbf{We prove that} $\alpha_1\cap\alpha_{B-3}'=\emptyset$. By Remark~\ref{rem:descalphaprime}, the description of $\alpha_{B-3}'$  as in Lemma~\ref{lem:descalpha} is obtained from the description of $\alpha_2$ by interchanging the digits $a\leftrightarrow B-1-a$ and the sets $K_1\leftrightarrow K_4,K_2\leftrightarrow K_5,K_3\leftrightarrow K_6$. We deduce from these descriptions that
$$\begin{array}{rcl}
\alpha_1&\subset&\bigcup_{k=B-A}^{B-1}\T_{10(B-1)k} \;\cup\; \bigcup_{k=0}^{A-2}0.1k[K_4]\;\cup\;\bigcup_{k=0}^{A-3}0.1k[K_5]\\\\
&&\cup\;\T_{1(A-2)(B-2)0}\;\cup\;\bigcup_{k=0}^{A-1}\T_{1(A-2)(B-1)k}
\end{array}
$$
and 
$$\begin{array}{rcl}
\alpha_{B-3}'&\subset&\bigcup_{k=0}^{A-1}\T_{2(B-1)0k}\;\cup\;\bigcup_{k=B-A+1}^{B-1}0.2k[K_1]\;\cup\;\bigcup_{k=B-A+2}^{B-1}0.2k[K_2]\\\\
&&\cup\;\T_{2(B-A+1)1(B-1)}\;\cup\;\bigcup_{k=B-A}^{B-1}\T_{2(B-A+1)0k}.
\end{array}
$$
We consider the sets $\T_{a_1'a_2'a_3'a_4'}$  and $\T_{a_1a_2a_3a_4}$ showing up in the right side of the above inclusions, as well as sets of this form with digits $a_1',\ldots,a_4'$ and $a_1,\ldots,a_4$ corresponding to the sets $0.1k[K]$ and $0.2k[K]$. If two such sets have nonempty intersection, they must satisfy Lemma~\ref{lem:subd}. In particular, since $a_1-a_1'=2-1=1$, we must have $a_2-a_2'\in\{A,A-1,A-2\}$. 

If $a_2-a_2'=A$, then $a_3=B-1$ by Lemma~\ref{lem:subd}, but this does not occur, since $a_3$ is at most $A-1$ in the above sets. 

If $a_2-a_2'=A-1$, then $a_3-a_3'\in\{B-A,B-A+1,B-A+2\}$. We consider the different cases. 
\begin{itemize}
\item $a_3-a_3'=B-A$ ($\geq 2$). If  $a_3\in\{0,1\}$ or if $a_3'\in\{B-2,B-1\}$, then $a_3-a_3'\leq 1$. Hence it remains to check that
$$\bigcup_{k=0}^{A-3}0.1k[K_5]\;\cap\;\bigcup_{k=B-A+2}^{B-1}0.2k[K_2]=\emptyset. 
$$ 
Here, $a_3\leq A-1$ and $a_3'\geq B-A$. Thus, if $a_3\leq A-2$ or if $a_3'\geq B-A+1$, then $a_3-a_3'\leq 1$. Hence we are left to check that 
$$0.1a_2'(B-A)K_1\;\cap\;0.2a_2(A-1)K_4=\emptyset.
$$
However, by Lemma~\ref{lem:subd}, for this intersection to be nonempty, we should have $a_4=0$, which does not hold. 
\item $a_3-a_3'\in\{B-A+1,B-A+2\}$. Note that $B-A+1\geq 3$ and $B-A+1\geq 4$. However, since $a_3\leq A-1$ and $a_3'\geq B-A$, we have $a_3-a_3'\leq 2$. Therefore, these cases do not occur.
\end{itemize}
Finally, if $a_2-a_2'=A-2$ then, by Lemma~\ref{lem:subd}, $a_3-a_3'=-1$ and the remaining digit sequences $(a_n)_{n\geq 4}$, $(a_n')_{n\geq 4}$ are the periodic sequences $(0,B-1,0,\ldots)$, $(B-1,0,B-1,\ldots)$. In particular, the condition $a_4=0$ is only satisfied by the set $\T_{2(B-1)00}$ and the condition $a_4'=B-1$ only satisfied by the set $\T_{10(B-1)(B-1)}$. But then $a_2-a_2'=B-1\ne A-2$. 


\bibliography{biblio}
\bibliographystyle{siam}

\end{document}